\documentclass[10pt]{article}
\usepackage[top=1in,left=0.9in,right = 0.6in, bottom=0.6in, footskip = 0.2in]{geometry}
\usepackage[english]{babel}
\usepackage{algorithm,algorithmicx}
\usepackage{tcolorbox}
\usepackage[normalem]{ulem}
\usepackage{tikz}

\usepackage{framed}
\usepackage{algpseudocode}
\usepackage{amsthm,nccmath}
\usepackage{amsmath,bm}
\usepackage[makeroom]{cancel}
\usepackage{bigints}
\usepackage{amssymb}
\usepackage{enumitem}
\usepackage{graphicx}
\usepackage{graphics}
\usepackage{tabularx}
\usepackage{multirow}
\usepackage[font=small,labelfont=bf]{caption}
\usepackage[font=small]{subcaption}
\usepackage{float}
\usepackage{epstopdf}
\usepackage{footnote}
\makesavenoteenv{tabular}
\makesavenoteenv{table}
\usepackage{bbm}
\usepackage{amsthm}
\usepackage{stmaryrd}
\usepackage{amsmath,bm}
\usepackage{amssymb}
\usepackage{mathtools, cuted}
\usepackage{kantlipsum,setspace}
\usepackage[normalem]{ulem}

\newcommand{\pk}[1]{\textcolor{black}{#1}}

\usepackage{dblfloatfix}
\theoremstyle{plain}
\newtheorem{theorem}{Theorem}[section]
\newtheorem{lemma}[theorem]{Lemma}
\newtheorem*{lemma*}{Lemma}

\theoremstyle{definition}
\newtheorem{defn}{Definition}[section]
\newtheorem*{defn*}{Definition}

\theoremstyle{remark}

\DeclareMathOperator*{\argmax}{arg\,max}

\usepackage{tikz}
\usetikzlibrary{calc,trees,positioning,arrows,chains,shapes.geometric,%
	decorations.pathreplacing,decorations.pathmorphing,shapes,%
	matrix,shapes.symbols}

\tikzstyle{startstop} = [rectangle, draw, rounded corners, align=center, minimum width=3cm, minimum height=1cm,text centered]
\tikzstyle{decision} = [diamond, draw, fill=blue!20, 
text width=4.5em, text badly centered, node distance=3cm, inner sep=0pt]
\tikzstyle{block} = [rectangle, draw, fill=blue!10, align=center, rounded corners, minimum width=3cm, minimum height=1cm]
\tikzstyle{blockcast} = [rectangle, draw, fill=red!10, align=center, rounded corners, minimum width=3cm, minimum height=0.45cm]
\tikzstyle{line} = [draw, -latex']
\tikzstyle{cloud} = [draw, ellipse,fill=red!20, node distance=3cm,
minimum height=2em]

\usepackage{hyperref}

\title{Parallel Restarted SPIDER - Communication Efficient Distributed Nonconvex Optimization with Optimal Computation Complexity}

\author{Pranay Sharma$^{*}$, Swatantra Kafle$^{*}$, Prashant Khanduri$^{*}$, \\
Saikiran Bulusu$^{*}$, Ketan Rajawat$^{\dagger}$ and Pramod K. Varshney$^{*}$\\
$^{*}$Department of Electrical Engineering and Computer Science,\\
Syracuse University, Syracuse, New York-13244\\ $^{\dagger}$Department of Electrical Engineering, \\
IIT Kanpur, Kanpur, India-208016.\\
\{psharm04, skafle, pkhandur, sabulusu, varshney\}@syr.edu, and ketan@iitk.ac.in.}

\begin{document}
\maketitle
\begin{abstract}
In this paper, we propose a distributed algorithm for stochastic smooth, non-convex optimization. We assume a worker-server architecture where $N$ nodes, each having $n$ (potentially infinite) number of samples, collaborate with the help of a central server to perform the optimization task. The global objective is to minimize the average of local cost functions available at individual nodes. The proposed approach is a non-trivial extension of the popular parallel-restarted SGD algorithm, incorporating the optimal variance-reduction based SPIDER gradient estimator into it. We prove convergence of our algorithm to a first-order stationary solution. The proposed approach achieves the best known communication complexity $O(\epsilon^{-1})$ along with the optimal computation complexity. For finite-sum problems (finite $n$), we achieve the optimal computation (IFO) complexity $O(\sqrt{Nn}\epsilon^{-1})$. For online problems ($n$ unknown or infinite), we achieve the optimal IFO complexity $O(\epsilon^{-3/2})$. In both the cases, we maintain the linear speedup achieved by existing methods. This is a massive improvement over the $O(\epsilon^{-2})$ IFO complexity of the existing approaches. Additionally, our algorithm is general enough to allow non-identical distributions of data across workers, as in the recently proposed federated learning paradigm.
\end{abstract}



\section{Introduction}
The current age of Big Data is built on the foundation of distributed systems, and efficient distributed algorithms to run on these systems \cite{Xing2016strategies}. With the rapid increase in the volume of the data being fed into these systems, storing and processing all this data at a central location becomes infeasible. Such a central \textit{server} requires a gigantic amount of computational and storage resources \cite{Xing2015Petuum}. Also, the input data is usually collected from a myriad of sources, which might inherently be distributed \cite{konevcny16federated}. Transferring all this data to a single location might be expensive. Depending on the sensitivity of the underlying data, this might also lead to concerns about maintaining the privacy and anonymity of this data \cite{leaute2013protecting}. Therefore, even in situations where it is possible to have central servers, it is not always desirable.

One of the factors which has made this Big Data explosion possible is the meteoric rise in the capabilities, and the consequent proliferation, of end-user devices \cite{konevcny16federated}. These \textit{worker} nodes or machines have significant storage and computational resources. One simple solution to the central server problem faced by distributed systems is that the server offloads some of its conventional tasks to these workers \cite{Xing2015Petuum}. More precisely, instead of sharing its entire data with the server, which is expensive communication-wise, each worker node carries computations on its data itself. The results of these ``local'' computations are then shared with the server. The server aggregates them to arrive at a ``global'' estimate, which is then shared with all the workers. The following distributed optimization problem, which we solve in this paper, obeys this governing principle.

\subsection{Problem}
\begin{align}
\min_{x \in \mathbb{R}^d} f(\mathbf{x}) \triangleq \frac{1}{N} \sum_{i=1}^N f_i(\mathbf{x}), \label{eq_problem}
\end{align}
where $N$ is the number of worker nodes. The local function corresponding to node $i$, $f_i (\mathbf{x})$ is a smooth, potentially non-convex function. We consider two variants of this problem.
\begin{itemize}
    \item \textbf{Finite-Sum Case:} Each individual node function $f_i$ in \eqref{eq_problem} is an empirical mean of function values corresponding to $n$ samples. Therefore, the problem is of the form
    \begin{align}
        \min_{x \in \mathbb{R}^d} f(\mathbf{x}) \triangleq \frac{1}{N} \sum_{i=1}^N \frac{1}{n} \sum_{j=1}^n f_i(\mathbf{x}, \xi_j), \label{eq_problem_finite}
    \end{align}
    where $\xi_j$ denotes the $j$-th sample and $f_i(\mathbf{x}, \xi_j)$ is the cost corresponding to the $j$-th sample.
    \item \textbf{Online Case:} Each individual node function $f_i$ in \eqref{eq_problem} is an expected value. Hence, the problem has the form
    \begin{align}
        \min_{x \in \mathbb{R}^d} f(\mathbf{x}) \triangleq \frac{1}{N} \sum_{i=1}^N \mathbb{E}_{\xi \in \mathcal{D}_i} f_i(\mathbf{x}, \xi), \label{eq_problem_online}
    \end{align}
    where $\mathcal{D}_i$ denotes the distribution of samples at the $i$-th node. Note that $\mathcal{D}_i$ can be different across different workers. This scenario is the popular \textbf{federated learning} \cite{konevcny16federated} model.
\end{itemize}
Throughout the paper, we assume that given a point $\mathbf{y}$ each node can choose samples $\xi$ independently, and the stochastic gradients of these sample functions are unbiased estimators of the actual gradient.
\pk{Generally, the optimality of a nonconvex problem is measured in terms of an $\epsilon$-stationary point} which is defined as a point $\mathbf{x}$ such that
\begin{align*}
    \left\| \nabla f(\mathbf{x}) \right\|^2 \leq \epsilon.
\end{align*}
\pk{which is updated to 
$$ \mathbb{E} \left\| \nabla f(\mathbf{x}) \right\|^2 \leq \epsilon,$$ 
for stochastic algorithms, where the expectation is taken to account for the randomness introduced by the stochastic nature of the algorithms. The point $\mathbf{x}$ which satisfies the above is referred to as a first-order stationary ({\bf FoS}) point of \eqref{eq_problem}. Note that the above quantity encodes only  the gradient error and ignores the consensus error across different nodes. Since, the algorithm proposed in this work is a distributed algorithm, it is appropriate to include consensus error term in the definition of the {\bf FoS} point \cite{hong19arXiv}. For this purpose, we modify the definition of {\bf FoS} point and include consensus error term in the definition as:
\begin{defn}{$\epsilon$-\textbf{First-order stationary ($\epsilon$-FoS) point}} \cite{hong19arXiv}
Let $\{\mathbf{x}_i\}_{i = 1}^N$ with $\mathbf{x}_i \in \mathbb{R}^d$ be the local iterates at $N$ nodes and $\bar{\mathbf{x}} = \frac{1}{N}\sum_{i = 1}^N \mathbf{x}_i$, then we define the $\epsilon$-{\bf FoS} point $\bar{\mathbf{x}}$ as 
\begin{align*}
   \mathbb{E} \left\| \nabla f \left( \bar{\mathbf{x}} \right) \right\|^2 + \frac{1}{N} \sum_{i=1}^N \mathbb{E} \left\| \mathbf{x}_{i} - \bar{\mathbf{x}} \right\|^2  \leq \epsilon. 
\end{align*}
Note that the expectation is over the stochasticity of the algorithm. All our results are in terms of convergence to an $\epsilon$-\textbf{FoS} point.
\label{Def: FoS}
\end{defn}
Note that the above Definition \ref{Def: FoS} implies that an $\epsilon$-{\bf FoS} point will ensure that $  \mathbb{E} \left\| \nabla f \left( \bar{\mathbf{x}} \right) \right\|^2 \leq \epsilon$.} The complexity of the algorithm is measured in terms of communication and computation \pk{complexity, which are defined next.}
\begin{defn}{\textbf{Computation Complexity:}}
\label{Def: Comp}
In the Incremental First-order Oracle (IFO) framework \cite{agarwal14arxiv}, given a sample $\xi$ at node $i$ and a point $\mathbf{x}$, the oracle returns $(f_i(\mathbf{x}; \xi), \nabla f_i(\mathbf{x}; \xi))$. Each access to the oracle is counted as a single IFO operation. The sample or computation complexity of the algorithm is hence, the total (aggregated across nodes) number of IFO operations to achieve an $\epsilon$-FoS solution.
\end{defn}
\begin{defn}{\textbf{Communication Complexity:}}
\label{Def: Comm}
In one round of communication between the workers and the server, each worker sends its local vector $\in \mathbb{R}^d$ to the server, and receives an aggregated vector of the same dimension in return. The communication complexity of the algorithm is the number of such communication rounds required to achieve an $\epsilon$-FoS solution.
\end{defn}

\subsection{Related Work}

\subsubsection{Stochastic Gradient Descent (SGD)}
SGD is workhorse of the modern Big-Data machinery. Before moving to the discussion of the literature most relevant to our work, we provide a quick review of the most basic results using SGD. For strongly convex problems, to obtain an $\epsilon$-accurate solution, 
$O(1/ \epsilon)$ IFO-calls are required \cite{nemirovski09siam}. For general convex \cite{moulines11neurips} problems, to achieve an $\epsilon$-accurate solution, $O(1/\epsilon^2)$ IFO calls are required. For nonconvex problems \cite{ghadimi13siam}, to achieve an $\epsilon$-stationary solution, $O(1/\epsilon^2)$ IFO calls are required. In the absence of additional assumptions on the smoothness of stochastic functions, this bound is tight \cite{Carmon2019lower_Arxiv}. With additional assumptions, this bound can be improved using variance reduction methods, which we discuss in Section \ref{subsec_var_red}.

\subsubsection{Distributed Stochastic Gradient Descent (SGD)}
There is a vast and ever-growing body of literature on distributed SGD. However, here we limit our discussion almost exclusively to work in the domain of stochastic nonconvex optimization. A classical method to solve \eqref{eq_problem} is Parallel Mini-batch SGD \cite{dekel12jmlr}, \cite{li14neurips}. Each worker, in parallel, samples a (or a mini-batch of) local stochastic gradient(s). It uses this gradient estimator to update its local iterate and sends the latter to the server. The server node aggregates the local iterates, computes their average and broadcasts this average to all the workers. The workers and the server repeat the same process iteratively.\footnote{Alternatively, the worker nodes might send their local gradient estimators to the server, which then averages these, updates its iterate using the average, and broadcasts the new iterate to the workers.} The approach achieves an $\epsilon$-FoS point, with a linear speedup with the number of workers. This is because the total IFO complexity $O(\epsilon^{-2})$ is independent of the number of workers $N$. Hence each node only needs to compute $O(\frac{1}{N} \epsilon^{-2})$ gradients. However, the exchange of gradients and iterates between the workers and the server at each iteration results in a significant communication requirement, leading to communication complexity of $O(\frac{1}{N} \epsilon^{-2})$. To alleviate the communication cost, several approaches have recently been proposed. These include communicating compressed gradients to the server, as in quantized SGD \cite{wen17neurips_terngrad, qsgd17neurips} or sparsified SGD \cite{dryden16mlhpc, aji17arxiv}. The motivation behind using quantized gradient vectors (or sparse approximations of actual gradients) is to reduce the communication cost, while not significantly affecting the convergence rate. The number of communication rounds, however, still remains the same.

\textbf{Model Averaging:} To cut back on the communication costs further, the nodes might decide to make the communication and the subsequent averaging infrequent. At one extreme of this idea is \textbf{one-shot averaging}, in which averaging happens only once. All the nodes run SGD locally, in parallel, throughout the procedure. At the end of the final iteration, the local iterates are averaged and the average is returned as output. However, this has been shown, in some non-convex problems, to have poor convergence \cite{zhang16arxiv}. More frequent averaging emerges as the obvious next option to explore. For \textbf{strongly convex} problems, it was shown in \cite{stich2018arxiv} that model averaging can also achieve linear speedup with $N$, as long as averaging happens at least every $I = \frac{1}{N \epsilon}$ iterations. Following this positive result, a number of works have achieved similar results for \textbf{non-convex} optimization. The IFO complexity is still $O(\epsilon^{-2})$. However, savings on communication costs are demonstrated.

The resulting class of algorithms is referred to as \textbf{Parallel Restarted SGD}. The entire algorithm is divided into epochs, each of length $I$. Within each epoch, all the $N$ nodes run SGD locally, in parallel. All the nodes begin each epoch at the same point, which is the average of the local iterates at the end of the previous epoch. Essentially, each node runs SGD in \textit{parallel} for $I$ iterations. At the end of these local iterations, each worker sends its solution to the server. The server returns the average of these local iterates to the workers, each of which \textit{restarts} the local iterations from this new point. Hence the name Parallel Restarted SGD. In \cite{restarted_sgd19aaai}, the proposed approach achieves communication savings using model averaging. However, it works under the additional assumption of all the stochastic gradients having bounded second moment. This assumption is relaxed in \cite{yu19momentum}, which achieves further reduction in the communication requirement by adding momentum to the vanilla SGD run locally at the nodes. The communication cost is again improved in \cite{yu19icml_batchsize}, where the authors used dynamic batch sizes. This is achieved using a two-step approach. In the first step, for the special class of non-convex problems satisfying the Polyak-Lojasiewicz (PL) condition, exponentially increasing batch sizes are used locally. The fastest known convergence, with linear speedup, is thus achieved using only $O(\log (\epsilon^{-1}))$ communication rounds. Next, for general nonconvex problems, the algorithm proposed in the first step is used as a subroutine, to achieve linear speedup, while using only $O(\epsilon^{-1} \log (N^{-1} \epsilon^{-1}) ) = \tilde{O} (\epsilon^{-1})$ communication rounds, where $\tilde{O}(\cdot)$ subsumes logarithmic factors.

\begin{table}[h]
\centering
\begin{tabular}{l|c|c|c}
\hline
\multirow{2}{*}{Reference} & \multicolumn{2}{c|}{Communication Complexity} & \multirow{2}{*}{IFO Complexity} \\
\cline{2-3}
 & Identical $f_i(\cdot)$ & Non-identical $f_i(\cdot)$ & \\
\hline
\cite{dekel12jmlr} & $O\left( \frac{1}{N \epsilon^2} \right)$ & NA & \multirow{5}{*}{$O \left( \frac{1}{\epsilon^2} \right)$} \\
\cite{restarted_sgd19aaai}* & $O\left( \frac{1}{\epsilon^{3/2}} \right)$ & $O\left( \frac{1}{\epsilon^{3/2}} \right)$ & \\
\cite{jiang18neurips} & $O\left( \frac{N^2}{\epsilon} \right)$ & $O\left( \frac{\sqrt{N}}{\epsilon^{3/2}} \right)$ &  \\
\cite{yu19momentum} & $O\left( \frac{N}{\epsilon} \right)$ & $O\left( \frac{1}{\epsilon^{3/2}} \right)$ &  \\
\cite{yu19icml_batchsize} & $O\left( \frac{1}{\epsilon} \log (\frac{1}{N \epsilon}) \right)$ & NA & \\
\hline
This Paper & {\color{red}$O\left( \frac{1}{\epsilon} \right)$} & {\color{red}$O\left( \frac{1}{\epsilon} \right)$} & {\color{red} $\min \left\{ \frac{\sqrt{Nn}}{\epsilon}, \frac{1}{\epsilon^{3/2}} \right\}$ } \\
\hline
\end{tabular}
\caption{The Communication and IFO complexity of different distributed SGD algorithms to reach a $\epsilon$-FoS point, for the stochastic smooth non-convex optimization problem. Note that not all approaches are applicable to the more general setting where the distributions at different nodes are non-identical. This setting captures the recently proposed federated learning paradigm \cite{konevcny16federated}. *The approach in \cite{restarted_sgd19aaai} requires the additional assumption that the gradients have bounded second moments.}
\end{table}

Other approaches based on infrequent averaging include LAG \cite{chen18neurips_lag}. In LAG, during every iteration, the server requests \textit{fresh} gradients only from a subset of the workers, while for the remaining workers, it reuses the gradients received in the previous iteration. However, the approach has only been explored in the deterministic gradient descent setting. In \cite{cadambe19icml}, redundancy is introduced in the training data to achieve communication savings. The authors of \cite{jaggi18localSGD} provide a comprehensive empirical study of distributed SGD, with focus on communication efficiency and generalization performance. To the best of our knowledge, no improvement in the $O(\epsilon^{-2})$ IFO complexity benchmark for distributed stochastic non-convex optimization has been reported thus far.

\subsubsection{Variance Reduction}
\label{subsec_var_red}
For the sake of simplicity of the discussion in this section, we assume that all the samples in the finite-sum problem \eqref{eq_problem_finite} or the online problem \eqref{eq_problem_online} are available at a single node. To solve the finite-sum problem \eqref{eq_problem_finite}, where each $f_{i} (\cdot; \xi_j)$ is $L$-smooth and potentially non-convex, gradient descent
(GD) and SGD are the two classical approaches. In terms of per-iteration complexity, these form the two extremes: GD entails computing the full gradient at each node in each iteration, i.e., $O(N \times n)$ operations, while SGD requires only $O(1)$ computations per iteration. Overall, to reach an $\epsilon$-FoS point $\bar{\mathbf{x}}$, GD requires $O(N n \epsilon^{-1})$ \cite{nesterov98book}, while SGD requires $O(\epsilon^{-2})$ gradient evaluations. For large values of $N \times n$ (or in situations where $n$ is infinite as in the online setting \eqref{eq_problem_online}), SGD is the only viable option.

Empirically, SGD has been observed to have good initial performance, but its progress slows down near the solution. One of the reasons behind this, is the variance inherent in the stochastic gradient estimator used. A number of variance reduction estimators, for example, SAGA \cite{saga14neurips} and SVRG \cite{svrg13neurips}, have been proposed in the literature to ameliorate this problem. The algorithms based on these estimators, compute full (or batch) gradients at regular intervals, to ``guide'' the progress made by the intermediate SGD steps. This interleaving of GD and SGD steps has resulted in significantly improved convergence rates, for problems with mean-squared smoothness property\footnote{See Assumption (A1) in Section \ref{subsec_not_assum}.}. The IFO complexity required for finite-sum problems was first improved to $O((N n)^{2/3} \epsilon^{-1})$ in \cite{nonc_svrg16icml, allen16icml}, and then further improved to $O((N n)^{1/2} \epsilon^{-1})$ in \cite[SPIDER]{spider18neurips}, \cite[SPIDERBoost]{spiderboost18neurips}, \cite[SNVRG]{snvrg18neurips}, \cite[SARAH]{nguyen19arxiv_optimal}. Moreover, it is proved in \cite{spider18neurips} that $O((N n)^{1/2} \epsilon^{-1})$ is the optimal complexity for problems where $N \times n \leq O(\epsilon^{-2})$. Similarly, for online problems, variance reduction methods first improved upon SGD to achieve the IFO complexity of $O(\epsilon^{-5/3})$ \cite[SCSG]{lei16_scsg}, which was again improved to $O(\epsilon^{-3/2})$ \cite[SPIDER]{spider18neurips}, \cite{spiderboost18neurips}. Quite recently in \cite{Carmon2019lower_Arxiv}, it was shown that the IFO complexity of $O(\epsilon^{-3/2})$ is optimal for online problems. 

\subsubsection{Distributed Stochastic Variance Reduction Methods}
The existing literature on distributed variance-reduction methods is almost exclusively focused on convex and strongly convex problems. Empirically, these methods have been shown to be promising \cite[AIDE]{reddi16aide}. Single node SVRG requires gradient estimators at each iteration to be unbiased. This is a major challenge for distributed variants of SVRG. The existing approaches try to bypass this by simulating sampling extra data \cite{lee17jmlr_distsvrg}, \cite[DANE]{shamir14dane}, \cite{wang17memory}. To the best of our knowledge, \cite{cen19arxiv_distsvrg} is the only work that avoids this additional sampling.


\subsection{Contributions}
In this paper, we propose a distributed variant of the SPIDER algorithm, Parallel Restarted SPIDER (PR-SPIDER), to solve the non-convex optimization problem \eqref{eq_problem}. Note that PR-SPIDER is a non-trivial extension of both SPIDER and parallel-restarted SGD. This is because we need to optimize both communication and computation complexities. Averaging at every step, or too often, negatively impacts the communication savings. Too infrequent averaging leads to error terms building up as we see in the analysis, which has adverse impact on convergence.
\begin{itemize}
    \item For the online setting \eqref{eq_problem_online}, our proposed approach achieves the optimal overall (aggregated across nodes) IFO complexity of $O \left( \frac{\sigma}{\epsilon^{3/2}} + \frac{\sigma^2}{\epsilon} \right)$. This result improves the long-standing best known result of $O(\sigma^2 \epsilon^{-2})$, while also achieving the linear speedup achieved in the existing literature. The communication complexity achieved $O(\epsilon^{-1})$ is also the best known in the literature. To the best of our knowledge, the only other work to achieve the same communication complexity is \cite{yu19icml_batchsize}.
    \item For the finite-sum problem \eqref{eq_problem_finite}, our proposed approach achieves the overall IFO complexity of $O(\sqrt{N n} \epsilon^{-1})$. For problems where $N \times n \leq O (\epsilon^{-2})$, this is the optimal result one can achieve, even if all the $N \times n$ functions are available at a single location \cite{spider18neurips}. At the same time, we also achieve the best known communication complexity $O(\epsilon^{-1})$. 
    \item Compared to several existing approaches which require the samples across nodes to follow the same distribution, our approach is more general in the sense that the data distribution across nodes may be different (the federated learning problem \cite{konevcny16federated}).
\end{itemize}

\subsection{Paper Organization}
The rest of the paper is organized as follows: in Section \ref{sec_spider_finite}, we discuss our approach to solve the finite-sum problem \eqref{eq_problem_finite}. We propose PR-SPIDER, a distributed, parallel variant of the SPIDER algorithm \cite{spider18neurips, spiderboost18neurips}, followed by its convergence analysis. In Section \ref{sec_spider_online}, we solve the online problem \eqref{eq_problem_online}. The algorithm proposed, and the accompanying convergence analysis builds upon the finite-sum approach and the analysis from the previous section. Finally, we conclude the paper in Section \ref{sec_conclude}. All the proofs are in the Appendices at the end.

\subsection{Notations and Assumptions}
\label{subsec_not_assum}
Given a positive integer $N$, the set $\{1,\hdots, N \}$ is denoted by the shorthand $[1:N]$. $\| \cdot \|$ denotes the vector $\ell_2$ norm. Boldface letters are used to denote vectors. \pk{We use $x \wedge y$ to denote the minimum of two numbers $x,y \in \mathbb{R}$.}

Following assumptions hold for the rest of the paper.
\begin{itemize}
\item[(A1)] \textbf{Lipschitz-ness:} All the functions are mean-squared $L$-smooth.\footnote{This assumption might seem stringent. However, we can always choose $L$ as the maximum Lipschitz constant corresponding to all the functions across all nodes.}
\begin{align}
    \mathbb{E}_{\xi} \left\| \nabla f (\mathbf{x}; \xi) - \nabla f (\mathbf{y}; \xi) \right\|^2 \leq L^2 \left\| \mathbf{x} - \mathbf{y} \right\|^2, \qquad \forall \ \mathbf{x}, \mathbf{y} \in \mathbb{R}^d.
\end{align}
\item[(A2)] All the nodes begin the algorithm from the same starting point $\mathbf{x}^0$.
\end{itemize}

\section{Parallel Restarted SPIDER - Finite Sum Case}
\label{sec_spider_finite}
We consider a network of $N$ worker nodes connected to a server node. The objective is to solve \eqref{eq_problem_finite}. Note that the number of sample functions at different nodes $i \neq j$, for $i,j \in [1:N]$, can be non-uniform, i.e., $n_i \neq n_j$. However, to ease the notational burden slightly, we assume $n_j = n$, $\forall \ j \in [1:N]$.

\subsection{Proposed Algorithm}
The proposed algorithm is inspired by the recently proposed SPIDER algorithm \cite{spider18neurips, spiderboost18neurips} for single-node stochastic nonconvex optimization. Like numerous variance-reduced approaches proposed in the literature, our algorithm also proceeds in epochs.

At the beginning of each epoch, each worker node has access to the same iterate, and the full gradient $\nabla f(\cdot)$ computed at this value. These are used to compute the first iterate of the epoch $\mathbf{x}_{i,1}^{s+1}, \forall \ i \in [1:N]$. This is followed by the inner loop (step \ref{alg1_inner_loop_begin}-\ref{alg1_inner_loop_end}). At iteration $t$ in $s$-th epoch, the worker nodes first compute an estimator of the gradient, $\mathbf{v}_{i,t}^{s+1}$ (step \ref{alg1_inner_loop_dir}). This estimator is computed iteratively, using the previous estimate $\mathbf{v}_{i,t-1}^{s+1}$, the current iterate $\mathbf{x}_{i,t}^{s+1}$, and the previous iterate $\mathbf{x}_{i,t-1}^{s+1}$. The sample set $\mathcal{B}_{i,t}^{s+1}$ of size $B$ is picked uniformly randomly at each node, and independent of the other nodes. Such an estimator has been proposed in the literature for single-node stochastic nonconvex optimization \cite{spider18neurips, spiderboost18neurips, sarah17icml, nguyen17arxiv}, and has even been proved to be optimal in certain regimes.

Using the gradient estimator, the worker node computes the next iterate $\mathbf{x}_{i,t+1}^{s+1}$. This process is repeated $m-1$ times. Once every $I$ iterations of the inner loop (whenever $t$ mod $I=0$), the nodes send their local iterates and gradient estimators $\{ \mathbf{x}_{i,t}^{s+1}, \mathbf{v}_{i,t}^{s+1} \}_{i=1}^N$ to the server node. The server, in turn, computes their averages and returns the averages $\{ \bar{\mathbf{x}}_{t}^{s+1}, \bar{\mathbf{v}}_{t}^{s+1} \}$ to all the nodes (steps \ref{alg1_inner_avg_begin}-\ref{alg1_inner_avg_end}). The next iteration at each node proceeds using this iterate and direction. 

At the end of the inner loop $(t=m)$, all the worker nodes send their local iterates $\{ \mathbf{x}_{i,m}^{s+1} \}_{i=1}^N$ to the server. The server computes the model average $\bar{\mathbf{x}}_{m}^{s+1}$, and returns it to all the workers (steps \ref{alg1_x_avg_comp}-\ref{alg1_x_avg_assign}). The workers compute the full gradient of their respective functions $\{ \nabla f_i (\bar{\mathbf{x}}_{m}^{s+1}) \}_{i=1}^N$ at this point, and send it to the server. The server averages these, and returns this average (which is essentially $\nabla f (\bar{\mathbf{x}}_{m}^{s+1})$) to the worker nodes (steps \ref{alg1_v_avg_comp}-\ref{alg1_v_avg_assign}). Consequently, all the worker nodes start the next epoch at the same point, and along the same descent direction. This ``restart'' of the local computation is along the lines of Parallel-Restarted SGD \cite{stich2018arxiv, restarted_sgd19aaai}.

\pk{Before we proceed with the proof of convergence of PR-SPIDER for the finite-sum problem, following is the organization of the proof \eqref{eq_problem_finite}. 
\pk{\subsubsection*{Organization of the Proof:} 
In Section \ref{sec_prelim}, we first present some preliminary results required to prove the main result of Section \ref{sec_exp_f} given in Theorem \ref{thm_conv_finite_sum}. We first present Lemma \ref{lemma_finite_bd_var}, which bounds the variance of the average gradient estimates, $\bar{\mathbf{v}}_t^{s + 1}$}, across all agents. Proving Lemma \ref{lemma_finite_bd_var} requires an intermediate result given in Lemma \ref{lemma_finite_nw_error}, which bounds the network disagreements across gradient estimates and local iterates. In Section \ref{sec_exp_f}, we first prove the descent lemma given in Lemma \ref{lemma_finite_descent} using the results presented in Section \ref{sec_prelim}. Finally, the main result of the section is presented in Theorem \ref{thm_conv_finite_sum}, which uses the preliminary results provided in Section \ref{sec_prelim} and Lemma \ref{lemma_finite_descent}.}

\begin{algorithm}[h!]
\caption{PR-SPIDER - Finite Sum Case}
\label{alg1}
\begin{algorithmic}[1]
	\State{\textbf{Input:} Initial iterate $\mathbf{x}_{i,m}^0 = \mathbf{x}^0, \mathbf{v}_{i,m}^0 = \nabla f \left( \mathbf{x}^0 \right) \forall \ i \in [1:N]$}
	\For{$s=0$ to $S-1$}
    	\State{$\mathbf{x}^{s+1}_{i,0} = \mathbf{x}_{i,m}^s, \forall \ i \in [1:N]$ \label{alg1_x_init}}
    	\State{$\mathbf{v}^{s+1}_{i,0} = \mathbf{v}_{i,m}^s, \forall \ i \in [1:N]$ \label{alg1_v_init}}
    	\State{$\mathbf{x}_{i,1}^{s+1} = \mathbf{x}_{i,0}^{s+1} - \gamma \mathbf{v}_{i,0}^{s+1}$}
	    \For{$t=1$ to $m-1$ \label{alg1_inner_loop_begin}}
        	    \State{$\mathbf{v}_{i,t}^{s+1} = \frac{1}{B} \sum_{\xi_{i,t}^{s+1} \in \mathcal{B}_{i,t}^{s+1}} \left[ \nabla f_i \big( \mathbf{x}_{i,t}^{s+1}; \xi_{i, t}^{s+1} \big) - \nabla f_i \big( \mathbf{x}_{i,t-1}^{s+1}; \xi_{i, t}^{s+1} \big) \right] + \mathbf{v}_{i,t-1}^{s+1}, \forall \ i \in [1:N]$ \qquad ($|\mathcal{B}_{i,t}^{s+1}| = B$) \label{alg1_inner_loop_dir}}
    	    \If{$t$ mod $I = 0$ \label{alg1_inner_avg_begin}}
    	        \State{$\mathbf{x}_{i,t}^{s+1} = \bar{\mathbf{x}}_{t}^{s+1} \triangleq \frac{1}{N} \sum_{j=1}^N \mathbf{x}_{j,t}^{s+1}, \forall \ i \in [1:N]$ \label{alg1_x_avg_inloop}}
    	        \State{$\mathbf{v}_{i,t}^{s+1} = \bar{\mathbf{v}}_{t}^{s+1} \triangleq \frac{1}{N} \sum_{j=1}^N \mathbf{v}_{j,t}^{s+1}, \forall \ i \in [1:N]$ \label{alg1_v_avg_inloop}}
    	   \EndIf \label{alg1_inner_avg_end}
    	    \State{$\mathbf{x}_{i,t+1}^{s+1} = \mathbf{x}_{i,t}^{s+1} - \gamma \mathbf{v}_{i,t}^{s+1}, \forall \ i \in [1:N]$ \label{alg1_inner_loop_iter}}
	    \EndFor \label{alg1_inner_loop_end}
	    \If{$s < S-1$}
    	    \State{$\bar{\mathbf{x}}_{m}^{s+1} = \frac{1}{N} \sum_{j=1}^N \mathbf{x}_{j,m}^{s+1}$ \label{alg1_x_avg_comp}}
    	    \State{$\mathbf{x}_{i,m}^{s+1} = \bar{\mathbf{x}}_{m}^{s+1}, \forall \ i \in [1:N]$ \label{alg1_x_avg_assign}}
    	    \State{$\bar{\mathbf{v}}_{m}^{s+1} = \frac{1}{N} \sum_{j=1}^N \nabla f_j \left( \mathbf{x}_{j,m}^{s+1} \right) = \nabla f \left( \bar{\mathbf{x}}_{m}^{s+1} \right)$\label{alg1_v_avg_comp}}
    	    \State{$\mathbf{v}_{i,m}^{s+1} = \bar{\mathbf{v}}_{m}^{s+1}, \forall \ i \in [1:N]$ \label{alg1_v_avg_assign}}
	    \EndIf
	\EndFor
	\State{\textbf{Return}}
\end{algorithmic}
\end{algorithm}

\subsection{Preliminaries}
\label{sec_prelim}
We begin by first stating a few preliminary results which shall be useful in the convergence proof. First, we bound the error in the average (across nodes) gradient estimator.
\begin{lemma}{(Gradient Estimate Error)}
\label{lemma_finite_bd_var}
For $0 < t < m, 0 \leq s \leq S-1$, the sequence of iterates $\{ \mathbf{x}_{i,t}^{s+1} \}_{i,t}$ and $\{ \mathbf{v}_{i,t}^{s+1} \}_{i,t}$ generated by Algorithm \ref{alg1} satisfies 
\begin{align}
    & \mathbb{E} \left\| \bar{\mathbf{v}}_{t}^{s+1} - \mfrac{1}{N} \sum_{i=1}^N \nabla f_i \left( \mathbf{x}_{i,t}^{s+1} \right) \right\|^2 \leq \underbrace{\mathbb{E} \left\| \bar{\mathbf{v}}_{0}^{s+1} - \mfrac{1}{N} \sum_{i=1}^N \nabla f_i \left( \mathbf{x}_{i,0}^{s+1} \right) \right\|^2}_{E_0^{s+1}} + \frac{L^2}{N^2 B} \sum_{i=1}^N \sum_{\ell = 0}^{t-1} \mathbb{E} \Big\| \mathbf{x}_{i,\ell + 1}^{s+1} - \mathbf{x}_{i,\ell}^{s+1} \Big\|^2 \label{eq_lemma_finite_1}
\end{align}
\end{lemma}
\begin{proof}
See Appendix \ref{app_lemma_finite_bd_var}.
\end{proof}
The error in the average gradient estimate at time $t$ is bounded in terms of the corresponding error at the beginning of the epoch, and the average cumulative difference of consecutive local iterates. In \eqref{eq_lemma_finite_1}, we explicitly define $E_0^{s+1} \triangleq \mathbb{E} \left\| \bar{\mathbf{v}}_{0}^{s+1} - \mfrac{1}{N} \sum\nolimits_{i=1}^N \nabla f_i \left( \mathbf{x}_{i,0}^{s+1} \right) \right\|^2$. Note that by Algorithm \ref{alg1} (for finite-sum problems),
\begin{align*}
    \bar{\mathbf{v}}_{0}^{s+1} &= \mfrac{1}{N} \sum\nolimits_{i=1}^N \mathbf{v}_{i,0}^{s+1} = \mfrac{1}{N} \sum\nolimits_{i=1}^N \mathbf{v}_{i,m}^{s} = \bar{\mathbf{v}}_{m}^{s} \tag*{(steps \ref{alg1_v_init}, \ref{alg1_v_avg_assign} in Algorithm \ref{alg1})} \\
    &= \mfrac{1}{N} \sum\nolimits_{i=1}^N \nabla f_j \left( \bar{\mathbf{x}}_{m}^{s} \right) \tag*{(steps \ref{alg1_x_avg_comp}-\ref{alg1_v_avg_comp} in Algorithm \ref{alg1})} \\
    &= \mfrac{1}{N} \sum\nolimits_{i=1}^N \nabla f_j \left( \bar{\mathbf{x}}_{j,0}^{s+1} \right) \tag*{(step \ref{alg1_x_init} in Algorithm \ref{alg1})}
\end{align*}
Therefore, $E_0^{s+1} = 0, \forall s$. However, we retain it (as in \cite{hong19arXiv}), as they shall be needed in the analysis of online problems.

Next, we bound the network disagreements of the local estimates relative to global averages. There are two such error terms, corresponding to: 1) the local gradient estimators, and 2) the local iterates. For this purpose, we first define
\begin{equation}
\label{eq_tau}
    \tau(\ell) = \begin{cases}
        \argmax_j \left\{ j \mid j < \ell, j \text{ mod } I = 0  \right\} & \text{ if } \ell \text{ mod } I \neq 0 \\ 
        \ell & \text{ otherwise.}
        \end{cases}
\end{equation}
Note that, $\tau(\ell)$ is the largest iteration index smaller than (or equal to) $\ell$, which is a multiple of $I$. Basically, looking back from $\ell$, $\tau(\ell)$ is the latest time index when averaging happened in the current epoch (steps \ref{alg1_x_avg_inloop}-\ref{alg1_v_avg_inloop}). 
\begin{lemma}{(Network Disagreements)}
\label{lemma_finite_nw_error}
Given $0 \leq \ell \leq m$, $\alpha > 0, \delta > 0, \theta > 0$ such that $\theta = \delta + 8 \gamma^2 L^2 ( 1 + \frac{1}{\delta} )$. For $\ell \text{ mod } I \neq 0$,
\begin{align}
    & \sum_{i=1}^N \mathbb{E} \left\| \mathbf{v}_{i,\ell}^{s+1} - \bar{\mathbf{v}}_{\ell}^{s+1}  \right\|^2 \leq 8 \gamma^2 N L^2 \big( 1 + \mfrac{1}{\delta} \big) \sum_{j=\tau(\ell)}^{\ell-1} (1+\theta)^{\ell-1-j} \mathbb{E} \left\| \bar{\mathbf{v}}_{j}^{s+1}  \right\|^2 \label{eq_lemma_finite_nw_error_1} \\
    & \sum_{i=1}^N \mathbb{E} \left\| \mathbf{x}_{i,\ell}^{s+1} - \bar{\mathbf{x}}_{\ell}^{s+1} \right\|^2 \leq \left( 1 + \mfrac{1}{\alpha} \right) \gamma^2 \sum_{i=1}^N \sum_{j=\tau(\ell)+1}^{\ell-1} (1+\alpha)^{\ell-1-j} \mathbb{E} \left\| \mathbf{v}_{i,j}^{s+1} - \bar{\mathbf{v}}_{j}^{s+1} \right\|^2 \label{eq_lemma_finite_nw_error_2}
\end{align}
If $\ell \text{ mod } I = 0,$ $\sum_{i=1}^N \mathbb{E} \| \mathbf{v}_{i,\ell}^{s+1} - \bar{\mathbf{v}}_{\ell}^{s+1}  \|^2 = \sum_{i=1}^N \mathbb{E} \| \mathbf{x}_{i,\ell}^{s+1} - \bar{\mathbf{x}}_{\ell}^{s+1} \|^2 = 0$.
\end{lemma}
\begin{proof}
See Appendix \ref{app_lemma_finite_nw_error}.
\end{proof}
Lemma \ref{lemma_finite_nw_error} quantifies the network error at any time instant $\ell$. Owing to the frequent averaging (every $I$ iterations), this error build-up is limited. One can easily see that in the absence of within-epoch averaging (steps \ref{alg1_x_avg_inloop}-\ref{alg1_v_avg_inloop}), the sum over time in \eqref{eq_lemma_finite_nw_error_1} would start at $j=0$, rather than $j=\tau(\ell)$, leading to a greater error build-up. The same reasoning holds for \eqref{eq_lemma_finite_nw_error_2}.  Note that we can further bound \eqref{eq_lemma_finite_nw_error_2} by substituting the bound on $\sum_{i=1}^N \mathbb{E} \| \mathbf{v}_{i,j}^{s+1} - \bar{\mathbf{v}}_{j}^{s+1} \|^2$ from \eqref{eq_lemma_finite_nw_error_1}.

Next, we first present a descent lemma which along with the preliminary results of this section is then used to prove the convergence of PR-SPIDER for the finite sum problem \eqref{eq_problem_finite}.

\subsection{Convergence Analysis}
\label{sec_exp_f}
Given $\bar{\mathbf{x}}_{t+1}^{s+1} = \bar{\mathbf{x}}_{t}^{s+1} - \gamma \bar{\mathbf{v}}_{t}^{s+1}$, using $L$-Lipshcitz gradient property of $f$
\begin{align}
\mathbb{E} f \left( \bar{\mathbf{x}}_{t+1}^{s+1} \right) & \leq \mathbb{E} f \left( \bar{\mathbf{x}}_{t}^{s+1} \right) - \gamma \mathbb{E} \left\langle \nabla f \left( \bar{\mathbf{x}}_{t}^{s+1} \right), \bar{\mathbf{v}}_{t}^{s+1} \right\rangle + \frac{\gamma^2 L}{2} \mathbb{E} \left\| \bar{\mathbf{v}}_{t}^{s+1} \right\|^2 \nonumber \\
&= \mathbb{E} f \left( \bar{\mathbf{x}}_{t}^{s+1} \right) - \frac{\gamma}{2} \mathbb{E} \left\| \nabla f \left( \bar{\mathbf{x}}_{t}^{s+1} \right) \right\|^2 - \frac{\gamma}{2} (1-L \gamma) \mathbb{E} \left\| \bar{\mathbf{v}}_{t}^{s+1} \right\|^2 \nonumber \\
& \qquad \qquad + \frac{\gamma}{2} \mathbb{E} \left\| \nabla f \left( \bar{\mathbf{x}}_{t}^{s+1} \right) - \bar{\mathbf{v}}_{t}^{s+1} \right\|^2 \label{eq_upbd_f_avg_Lipschitz}
\end{align}
where in \eqref{eq_upbd_f_avg_Lipschitz}, we use $\left\langle a,b \right\rangle = \mfrac{||a||^2 + ||b||^2 - ||a-b||^2}{2}$. Next, we upper bound the last term in \eqref{eq_upbd_f_avg_Lipschitz}. We assume $t > 0$ (for $t = 0, \mathbb{E} \left\| \nabla f \left( \bar{\mathbf{x}}_{0}^{s+1} \right) - \bar{\mathbf{v}}_{0}^{s+1} \right\|^2 = 0$ - steps \ref{alg1_v_init}, \ref{alg1_v_avg_comp} in Algorithm \ref{alg1}).
\begin{align}
& \mathbb{E} \left\| \nabla f \left( \bar{\mathbf{x}}_{t}^{s+1} \right) - \bar{\mathbf{v}}_{t}^{s+1} \right\|^2 = \mathbb{E} \left\| \mfrac{1}{N} \sum_{i=1}^N \left\{ \nabla f_i \left( \bar{\mathbf{x}}_{t}^{s+1} \right) - \nabla f_i \left( \mathbf{x}_{i,t}^{s+1} \right) + \nabla f_i \left( \mathbf{x}_{i,t}^{s+1} \right) - \mathbf{v}_{i,t}^{s+1} \right\} \right\|^2 \nonumber \\
& \qquad \overset{(b)}{\leq}  \left[ 2 \mathbb{E} \left\| \mfrac{1}{N} \sum_{i=1}^N \left\{ \nabla f_i \left( \bar{\mathbf{x}}_{t}^{s+1} \right) - \nabla f_i \left( \mathbf{x}_{i,t}^{s+1} \right) \right\} \right\|^2 + 2 \mathbb{E} \left\| \mfrac{1}{N} \sum_{i=1}^N \left\{ \mathbf{v}_{i,t}^{s+1} - \nabla f_i \left( \mathbf{x}_{i,t}^{s+1} \right) \right\} \right\|^2 \right] \label{eq_upbd_gradf_avgv_1} \\
& \qquad \overset{(c)}{\leq} \frac{2L^2}{N} \sum_{i=1}^N \mathbb{E} \left\| \mathbf{x}_{i,t}^{s+1} - \bar{\mathbf{x}}_{t}^{s+1} \right\|^2 + 2 \mathbb{E} \left\| \bar{\mathbf{v}}_{t}^{s+1} - \mfrac{1}{N} \sum\nolimits_{i=1}^N \nabla f_i \left( \mathbf{x}_{i,t}^{s+1} \right) \right\|^2 \label{eq_upbd_gradf_avgv_2}
\end{align}
where \eqref{eq_upbd_gradf_avgv_1}, \eqref{eq_upbd_gradf_avgv_2} follow from $\mathbb{E} \left\| \sum_{i=1}^n X_i \right\|^2 \leq n \sum_{i=1}^n \mathbb{E} \left\| X_i \right\|^2$ and the mean-squared $L$-smoothness assumption (A1). More precisely,
\begin{align*}
    & \mathbb{E} \left\| \frac{1}{N} \sum_{i=1}^N \left\{ \nabla f_i \left( \bar{\mathbf{x}}_{t}^{s+1} \right) - \nabla f_i \left( \mathbf{x}_{i,t}^{s+1} \right) \right\} \right\|^2 \\
    \quad & \leq \frac{1}{N} \sum_{i=1}^N \mathbb{E} \left\| \nabla f_i \left( \bar{\mathbf{x}}_{t}^{s+1} \right) - \nabla f_i \left( \mathbf{x}_{i,t}^{s+1} \right) \right\|^2 \qquad \tag*{(using Jensen's inequality)} \\
    \quad & = \frac{1}{N} \sum_{i=1}^N \mathbb{E} \left\| \nabla \mathbb{E}_{\xi} f_i \left( \bar{\mathbf{x}}_{t}^{s+1}; \xi \right) - \nabla \mathbb{E}_{\xi} f_i \left( \mathbf{x}_{i,t}^{s+1}; \xi \right) \right\|^2 \\
    \quad & \leq \frac{1}{N} \sum_{i=1}^N \mathbb{E} \mathbb{E}_{\xi} \left\| \nabla  f_i \left( \bar{\mathbf{x}}_{t}^{s+1}; \xi \right) - \nabla f_i \left( \mathbf{x}_{i,t}^{s+1}; \xi \right) \right\|^2 \qquad \tag*{(using Jensen's inequality)} \\
    \quad & \leq \frac{1}{N} \sum_{i=1}^N \mathbb{E} \left\| \bar{\mathbf{x}}_{t}^{s+1} - \mathbf{x}_{i,t}^{s+1} \right\|^2 \qquad \tag*{using (A1)}.
\end{align*}
The first term in \eqref{eq_upbd_gradf_avgv_2} is the network error of local iterates, which is upper bounded in Lemma \ref{lemma_finite_nw_error}. Using Lemma \ref{lemma_finite_bd_var} to bound the second term of \eqref{eq_upbd_gradf_avgv_2}, we get

\begin{align}
& \mathbb{E} \big\| \nabla f \left( \bar{\mathbf{x}}_{t}^{s+1} \right) - \bar{\mathbf{v}}_{t}^{s+1} \big\|^2 \leq \frac{2 L^2}{N} \sum_{i=1}^N \mathbb{E} \big\| \mathbf{x}_{i,t}^{s+1} - \bar{\mathbf{x}}_{t}^{s+1} \big\|^2 + 2 E_0^{s+1} + \frac{2 L^2}{N^2 B} \sum_{i=1}^N \sum_{\ell = 0}^{t-1} \mathbb{E} \big\| \mathbf{x}_{i,\ell + 1}^{s+1} - \mathbf{x}_{i,\ell}^{s+1} \big\|^2 \nonumber \\
& \quad \leq \frac{2 L^2}{N} \sum_{i=1}^N \mathbb{E} \big\| \mathbf{x}_{i,t}^{s+1} - \bar{\mathbf{x}}_{t}^{s+1} \big\|^2 + \frac{2 \gamma^2 L^2}{N^2 B} \sum_{i=1}^N \sum_{\ell=0}^{t-1} \mathbb{E} \big\| \mathbf{v}_{i,\ell}^{s+1} - \bar{\mathbf{v}}_{\ell}^{s+1} + \bar{\mathbf{v}}_{\ell}^{s+1} \big\|^2 + 2 E_0^{s+1} \label{eq_upbd_gradf_avgv_3} \\
& \quad \leq \frac{2 L^2}{N} \sum_{i=1}^N \mathbb{E} \big\| \mathbf{x}_{i,t}^{s+1} - \bar{\mathbf{x}}_{t}^{s+1} \big\|^2 + \frac{4 \gamma^2 L^2}{N^2 B} \sum_{i=1}^N \sum_{\ell=0}^{t-1} \Big[ \mathbb{E} \big\| \mathbf{v}_{i,\ell}^{s+1} - \bar{\mathbf{v}}_{\ell}^{s+1}  \big\|^2 + \mathbb{E} \big\| \bar{\mathbf{v}}_{\ell}^{s+1}  \big\|^2 \Big] + 2 E_0^{s+1} \label{eq_upbd_gradf_avgv_4}
\end{align}
where \eqref{eq_upbd_gradf_avgv_3} follows since $\mathbf{x}_{i,\ell + 1}^{s+1} = \mathbf{x}_{i,\ell}^{s+1} - \gamma \mathbf{v}_{i,\ell}^{s+1}$. The second term in \eqref{eq_upbd_gradf_avgv_4} implies that during the epoch, as $t$ increases, in the absence of any averaging/communication within an epoch, the error $\sum_{\ell=0}^{t-1} \mathbb{E} \| \mathbf{v}_{i,\ell}^{s+1} - \bar{\mathbf{v}}_{\ell}^{s+1} \|^2$ keeps building up. To check this rapid growth, we introduce within-epoch averaging every $I$ iterations (steps \ref{alg1_x_avg_inloop}-\ref{alg1_v_avg_inloop}). Using the upper bounds on the two network error terms in \eqref{eq_upbd_gradf_avgv_4}, from Lemma \ref{lemma_finite_nw_error}, we get
\begin{align}
	& \mathbb{E} \big\| \nabla f \left( \bar{\mathbf{x}}_{t}^{s+1} \right) - \bar{\mathbf{v}}_{t}^{s+1} \big\|^2 \nonumber \\
	& \quad \leq \frac{2 L^2}{N} \left( 1 + \mfrac{1}{\alpha} \right) \gamma^2 \sum_{\ell=\tau(t)+1}^{t-1} (1+\alpha)^{t-1-\ell} 8 \gamma^2 N L^2 \big( 1 + \mfrac{1}{\delta} \big) \sum_{j=\tau(\ell)}^{\ell-1} (1+\theta)^{\ell-1-j} \mathbb{E} \Big\| \bar{\mathbf{v}}_{j}^{s+1}  \Big\|^2 \nonumber \\
	& \quad + \frac{4 \gamma^2 L^2}{N^2 B} \sum_{\ell=0}^{t-1} 8 \gamma^2 N L^2 \big( 1 + \mfrac{1}{\delta} \big) \sum_{j=\tau(\ell)}^{\ell-1} (1+\theta)^{\ell-1-j} \mathbb{E} \Big\| \bar{\mathbf{v}}_{j}^{s+1}  \Big\|^2 + \frac{4 \gamma^2 L^2}{N^2 B} N \sum_{\ell = 0}^{t-1} \mathbb{E} \Big\| \bar{\mathbf{v}}_{\ell}^{s+1}  \Big\|^2 + 2 E_0^{s+1}. \label{eq_upbd_gradf_avgv_5}
\end{align}
We substitute this upper bound in \eqref{eq_upbd_f_avg_Lipschitz} to derive the following descent lemma on function values.

\begin{lemma}{(Descent Lemma)}
\label{lemma_finite_descent}
In each epoch $s \in [1,S]$,
\begin{align}
    \mathbb{E} f \left( \bar{\mathbf{x}}_{m}^{s+1} \right) & \leq \mathbb{E} f \left( \bar{\mathbf{x}}_{0}^{s+1} \right) - \frac{\gamma}{2} \sum_{t=0}^{m-1} \mathbb{E} \big\| \nabla f \left( \bar{\mathbf{x}}_{t}^{s+1} \right) \big\|^2 - \frac{\gamma}{2} (1-L \gamma) \sum_{t=0}^{m-1} \mathbb{E} \big\| \bar{\mathbf{v}}_{t}^{s+1} \big\|^2 \nonumber \\
    & \qquad \qquad + \sum_{t=0}^{m-1} \mathbb{E} \Big\| \bar{\mathbf{v}}_{t}^{s+1} \Big\|^2 \left[ \frac{64 \gamma^5 L^4 m}{N B \delta^2} + \frac{2 \gamma^{3} L^2 m}{N B} + \frac{256 \gamma^{5} L^4}{\delta^4} \right] . \label{eq_lemma_finite_descent}
\end{align}
\end{lemma}
\begin{proof}
See Appendix \ref{app_lemma_finite_descent}.
\end{proof}
Lemma \ref{lemma_finite_descent} quantifies the decay in function value across one epoch. Clearly, the extent of decay depends on the precise values of algorithm parameters $\gamma, \delta, B, m$.

Rearranging the terms in \eqref{eq_lemma_finite_descent}, and summing over epoch index $s$, we get
\begin{align}
& \frac{\gamma}{2} \sum_{t=0}^{m-1} \mathbb{E} \big\| \nabla f \left( \bar{\mathbf{x}}_{t}^{s+1} \right) \big\|^2 + \frac{\gamma}{2} \left( 1 - L \gamma - \left[ \frac{128 \gamma^{4} L^4 m}{N B \delta^2} + \frac{4 \gamma^{2} L^2 m}{N B} + \frac{512 \gamma^{4} L^4}{\delta^4} \right] \right) \sum_{t=0}^{m-1} \mathbb{E} \big\| \bar{\mathbf{v}}_{t}^{s+1} \big\|^2 \nonumber \\
& \quad \leq \left( \mathbb{E} f \left( \bar{\mathbf{x}}_{0}^{s+1} \right) - \mathbb{E} f \left( \bar{\mathbf{x}}_{m}^{s+1} \right) \right) \label{eq_favg_in_epoch2} \\
\Rightarrow & \frac{1}{T} \sum_{s=0}^{S-1} \sum_{t=0}^{m-1} \mathbb{E} \big\| \nabla f \left( \bar{\mathbf{x}}_{t}^{s+1} \right) \big\|^2 + \left( 1 - L \gamma - \left[ \frac{128 \gamma^{4} L^4 m}{N B \delta^2} + \frac{4 \gamma^{2} L^2 m}{N B} + \frac{512 \gamma^{4} L^4}{\delta^4} \right] \right) \frac{1}{T} \sum_{s=0}^{S-1}  \sum_{t=0}^{m-1} \mathbb{E} \big\| \bar{\mathbf{v}}_{t}^{s+1} \big\|^2 \nonumber \\
& \quad \leq \frac{2}{T \gamma} \sum_{s=0}^{S-1} \left( \mathbb{E} f \left( \bar{\mathbf{x}}_{0}^{s+1} \right) - \mathbb{E} f \left( \bar{\mathbf{x}}_{m}^{s+1} \right) \right) = \frac{2}{T \gamma} \left( \mathbb{E} f \left( \bar{\mathbf{x}}^{0} \right) - \mathbb{E} f \left( \bar{\mathbf{x}}_{m}^{S+1} \right) \right)  \nonumber \\
& \quad \leq \frac{2 \left( f(\mathbf{x}^0) - f_{\ast} \right)}{T \gamma}. \label{eq_favg_over_all_epcohs}
\end{align}
where, in \eqref{eq_favg_in_epoch2}, we have collected all the terms in \eqref{eq_lemma_finite_descent} with $\sum_{t=0}^{m-1} \mathbb{E} \big\| \bar{\mathbf{v}}_{t}^{s+1} \big\|^2$ on the left hand side. \eqref{eq_favg_in_epoch2} is then summed across all epochs $s=0,\hdots,S-1$, divided by $T$ and divided by $\frac{\gamma}{2}$ to get \eqref{eq_favg_over_all_epcohs}. $f_{\ast} = \min_x f(x)$. Note that in \eqref{eq_favg_over_all_epcohs}, for small enough, but constant step size $\gamma$, we can ensure that
\begin{align}
    & 1 - L \gamma - \left[ \frac{128 \gamma^{4} L^4 m}{N B \delta^2} + \frac{4 \gamma^{2} L^2 m}{N B} + \frac{512 \gamma^{4} L^4}{\delta^4} \right] \geq \frac{1}{2} \label{eq_coeff_avg_v}.
\end{align}
See Appendix \ref{app_gamma_choice} for a choice of $\gamma$ which satisfies \eqref{eq_coeff_avg_v}.


\pk{Next, using \eqref{eq_favg_over_all_epcohs} and \eqref{eq_coeff_avg_v} above, we prove the convergence of PR-SPIDER for the finite sample case \eqref{eq_problem_finite}.}
\begin{theorem}{(Convergence)}
\label{thm_conv_finite_sum}
\label{thm}
For the finite-sum problem under Assumptions (A1)-(A2), for small enough step size $0 < \gamma < \frac{1}{8IL}$, 
\begin{align}
    \min_{s,t} \left[ \mathbb{E} \left\| \nabla f \left( \bar{\mathbf{x}}_{t}^{s+1} \right) \right\|^2 + \frac{1}{N} \sum_{i=1}^N \mathbb{E} \left\| \mathbf{x}_{i,t}^{s+1} - \bar{\mathbf{x}}_{t}^{s+1} \right\|^2 \right] \leq \frac{2 \left( f(\mathbf{x}^0) - f_{\ast} \right)}{T \gamma}. \label{eq_final_bd_3}
\end{align}
\end{theorem}
\begin{proof}
See Appendix \ref{app_thm_conv_finite_sum}.
\end{proof}
The above result now can be directly used to compute the bounds on the sample complexity (see Definition \ref{Def: Comp}) and the communication complexity (see Definition \ref{Def: Comm}) of the proposed algorithm PR-SPIDER for the finite sum problem \eqref{eq_problem_finite}.
\subsection{Sample Complexity}
\label{subsec_samp_comp}
Number of iterations $T$ satisfies
\begin{align}
    & \frac{2 \left( f(\mathbf{x}^0) - f_{\ast} \right)}{T \gamma} = \epsilon \quad \Rightarrow \quad T = \frac{2 \left( f(\mathbf{x}^0) - f_{\ast} \right)}{\gamma \epsilon} = \frac{C I}{\epsilon},
\end{align}
for constants $C, I$. Then for $m = I \sqrt{N n}, B = \frac{1}{I} \sqrt{\frac{n}{N}}$
\begin{align*}
    & N \times \left( \left \lceil \frac{T}{m} \right \rceil \cdot n + T \cdot B \right) \leq N \times \left( \left( \frac{C I}{m \epsilon} + 1 \right) \cdot n + \frac{C I}{\epsilon} \cdot B \right) \\
    & \leq N \times \left( \frac{C I}{\epsilon} \cdot \left( \mfrac{n}{m} + B \right) + n \right) \\
    &= O \left(  \frac{\sqrt{N n}}{\epsilon} + N n \right) \\
    &= O \left(  \frac{\sqrt{N n}}{\epsilon} \right) \tag*{for $N \times n \leq O(\epsilon^{-2})$.}
\end{align*}
Note that this is the optimal sample complexity achieved for $N \times n \leq O(\epsilon^{-2})$ in the single node case \cite{spider18neurips}, \cite{spiderboost18neurips}. Each of the $N$ nodes in our case need to compute $O \left( \frac{1}{\epsilon}\sqrt{\frac{n}{N}} \right)$ sample stochastic gradients.

\subsection{Communication Complexity}
Since communication happens once every $I$ iterations, the communication complexity (the number of communication rounds) is
\begin{align*}
    \left \lceil \frac{T}{I} \right \rceil \leq 1 + \frac{C}{\epsilon} = O \left( \frac{1}{\epsilon} \right).
\end{align*}

\section{Parallel Restarted SPIDER - Online Case}
\label{sec_spider_online}
We consider a network of $N$ worker nodes connected to a server node. The objective is to solve \eqref{eq_problem_online}. Note that the distribution of samples at different nodes can potentially be different, i.e., $\mathcal{D}_i \neq \mathcal{D}_j$, for $i \neq j$.

\subsection{Proposed Algorithm}
The pseudo-code of the approach is given in Algorithm \ref{alg2}. In the following, we only highlight the steps which are different from Algorithm \ref{alg1}. The proposed algorithm is pretty much the same as Algorithm \ref{alg1}, except a few changes to account for the fact that for problem \eqref{eq_problem_online}, exact gradients can never be computed. Hence full gradient computations are replaced by batch stochastic gradient computation, where the batches are of size $n_b$. Again, batch sizes across nodes can be non-uniform. However, we avoid doing so for the sake of simpler notations.

At the end of the inner loop $(t=m)$, first the local iterates are averaged and returned to the workers (steps \ref{alg2_x_avg_comp}-\ref{alg2_x_avg_assign}). Next, the workers compute batch stochastic gradients of their respective functions $\{ \frac{1}{n_b} \sum_{\xi_i} \nabla f_i \left( \cdot; \xi_i \right) \}_{i=1}^N$ at the common point $\bar{\mathbf{x}}_{m}^{s+1}$, and send these to the server. The server averages these, and returns this average $\bar{\mathbf{v}}_{m}^{s+1}$ to the worker nodes (steps \ref{alg2_v_avg_comp}-\ref{alg2_v_avg_assign}). As in Algorithm \ref{alg1}, all the worker nodes start the next epoch at the same point, and along the same descent direction. However, unlike Algorithm \ref{alg1}, this direction is not $\nabla f (\bar{\mathbf{x}}_{m}^{s+1})$.

\begin{algorithm}[h!]
\caption{PR-SPIDER - Online Case}
\label{alg2}
\begin{algorithmic}[1]
	\State{\textbf{Input:} Initial iterate $\mathbf{x}_{i,m}^0 = \mathbf{x}^0, \mathbf{v}_{i,m}^0 = \frac{1}{N} \sum_{j=1}^N \frac{1}{n_b} \sum_{\xi_j} \nabla f_j \left( \mathbf{x}^0; \xi_j \right), \forall \ i \in [1:N]$}
	\For{$s=0$ to $S-1$}
    	\State{$\mathbf{x}^{s+1}_{i,0} = \mathbf{x}_{i,m}^s, \forall \ i \in [1:N]$ \label{alg2_x_init}}
    	\State{$\mathbf{v}^{s+1}_{i,0} = \mathbf{v}_{i,m}^s, \forall \ i \in [1:N]$ \label{alg2_v_init}}
    	\State{$\mathbf{x}_{i,1}^{s+1} = \mathbf{x}_{i,0}^{s+1} - \gamma \mathbf{v}_{i,0}^{s+1}$}
	    \For{$t=1$ to $m-1$ \label{alg2_inner_loop_begin}}
        	    \State{$\mathbf{v}_{i,t}^{s+1} = \frac{1}{B} \sum_{\xi_{i,t}^{s+1} \in \mathcal{B}_{i,t}^{s+1}} \left[ \nabla f_i \big( \mathbf{x}_{i,t}^{s+1}; \xi_{i, t}^{s+1} \big) - \nabla f_i \big( \mathbf{x}_{i,t-1}^{s+1}; \xi_{i, t}^{s+1} \big) \right] + \mathbf{v}_{i,t-1}^{s+1}, \forall \ i \in [1:N]$ \qquad ($|\mathcal{B}_{i,t}^{s+1}| = B$) \label{alg2_inner_loop_dir}}
    	    \If{$t$ mod $I = 0$ \label{alg2_inner_avg_begin}}
    	        \State{$\mathbf{x}_{i,t}^{s+1} = \bar{\mathbf{x}}_{t}^{s+1} \triangleq \frac{1}{N} \sum_{j=1}^N \mathbf{x}_{j,t}^{s+1}, \forall \ i \in [1:N]$ \label{alg2_x_avg_inloop}}
    	        \State{$\mathbf{v}_{i,t}^{s+1} = \bar{\mathbf{v}}_{t}^{s+1} \triangleq \frac{1}{N} \sum_{j=1}^N \mathbf{v}_{j,t}^{s+1}, \forall \ i \in [1:N]$ \label{alg2_v_avg_inloop}}
    	   \EndIf \label{alg2_inner_avg_end}
    	    \State{$\mathbf{x}_{i,t+1}^{s+1} = \mathbf{x}_{i,t}^{s+1} - \gamma \mathbf{v}_{i,t}^{s+1}, \forall \ i \in [1:N]$ \label{alg2_inner_loop_iter}}
	    \EndFor \label{alg2_inner_loop_end}
	    \If{$s < S-1$}
    	    \State{$\bar{\mathbf{x}}_{m}^{s+1} = \frac{1}{N} \sum_{j=1}^N \mathbf{x}_{j,m}^{s+1}$ \label{alg2_x_avg_comp}}
    	    \State{$\mathbf{x}_{i,m}^{s+1} = \bar{\mathbf{x}}_{m}^{s+1}, \forall \ i \in [1:N]$ \label{alg2_x_avg_assign}}
    	    \State{$\bar{\mathbf{v}}_{m}^{s+1} = \frac{1}{N} \sum_{i=1}^N \frac{1}{n_b} \sum_{\xi_i} \nabla f_i \left( \mathbf{x}_{i,m}^{s+1}; \xi_i \right)$\label{alg2_v_avg_comp}}
    	    \State{$\mathbf{v}_{i,m}^{s+1} = \bar{\mathbf{v}}_{m}^{s+1}, \forall \ i \in [1:N]$ \label{alg2_v_avg_assign}}
    	 \EndIf
	\EndFor
	\State{\textbf{Return}}
\end{algorithmic}
\end{algorithm}

\subsection{Additional Assumptions}
\label{subsec_assum}
\begin{itemize}
\item[(A3)] \textbf{Bounded Variance:} There exists constant $\sigma$ such that
	\begin{align}
		\mathbb{E}_{\xi} \left\| \nabla f_i \left( \mathbf{x}; \xi \right) - \nabla f \left( \mathbf{x} \right) \right\|^2 \leq \sigma^2, \quad \forall \ i \in [1:N]. \label{eq_bd_var}
	\end{align}
\end{itemize}

\subsection{Convergence Analysis}
\begin{lemma}{(Bounded Variance)}
\label{lemma_online_bd_var}
For $0 \leq s \leq S-1$, the sequences of iterates $\{ \mathbf{x}_{i,t}^{s+1} \}_{i,t}$ and $\{ \mathbf{v}_{i,t}^{s+1} \}_{i,t}$ generated by Algorithm \ref{alg1} satisfy 
\begin{align}
    E_0^{s+1} = \mathbb{E} \Big\| \bar{\mathbf{v}}_{0}^{s+1} - \mfrac{1}{N} \sum_{i=1}^N \nabla f_i \left( \mathbf{x}_{i,0}^{s+1} \right) \Big\|^2 \leq \frac{\sigma^2}{N n_b} .\label{eq_lemma_online_1}
\end{align}
\end{lemma}
\begin{proof}
See Appendix \ref{app_lemma_online_bd_var}.
\end{proof}

Compared to the finite-sum setting, the only difference now is that $E_0^{s+1} \neq 0$. Both, Lemma \ref{lemma_finite_bd_var} and \ref{lemma_finite_nw_error} hold true. Next we state the descent lemma for the online case. 

\begin{lemma}{(Descent Lemma)}
\label{lemma_online_descent}
In each epoch $s \in [1,S]$,
\begin{align}
    \mathbb{E} f \left( \bar{\mathbf{x}}_{m}^{s+1} \right) & \leq \mathbb{E} f \left( \bar{\mathbf{x}}_{0}^{s+1} \right) - \frac{\gamma}{2} \sum_{t=0}^{m-1} \mathbb{E} \big\| \nabla f \left( \bar{\mathbf{x}}_{t}^{s+1} \right) \big\|^2 - \frac{\gamma}{2} (1-L \gamma) \sum_{t=0}^{m-1} \mathbb{E} \big\| \bar{\mathbf{v}}_{t}^{s+1} \big\|^2 \nonumber \\
    & + \sum_{t=0}^{m-1} \mathbb{E} \Big\| \bar{\mathbf{v}}_{t}^{s+1} \Big\|^2 \left[ \frac{64 \gamma^5 L^4 m}{N B \delta^2} + \frac{2 \gamma^{3} L^2 m}{N B} + \frac{256 \gamma^{5} L^4}{\delta^4} \right] + \frac{\gamma \sigma^2 m}{N n_b}. \label{eq_lemma_online_descent}
\end{align}
\end{lemma}
\begin{proof}
See Appendix \ref{app_lemma_online_descent}.
\end{proof}
Note that $\frac{\gamma \sigma^2 m}{N n_b}$ is the only extra term compared to \eqref{eq_lemma_finite_descent}. Rearranging the terms in \eqref{eq_lemma_online_descent}, and summing over epoch index $s$, analogous to \eqref{eq_favg_in_epoch2} and \eqref{eq_favg_over_all_epcohs} we get
\begin{align}
& \frac{\gamma}{2} \sum_{t=0}^{m-1} \mathbb{E} \big\| \nabla f \left( \bar{\mathbf{x}}_{t}^{s+1} \right) \big\|^2 + \frac{\gamma}{2} \left( 1 - L \gamma - \left[ \frac{128 \gamma^{4} L^4 m}{N B \delta^2} + \frac{4 \gamma^{2} L^2 m}{N B} + \frac{512 \gamma^{4} L^4}{\delta^4} \right] \right) \sum_{t=0}^{m-1} \mathbb{E} \big\| \bar{\mathbf{v}}_{t}^{s+1} \big\|^2 + \frac{\gamma \sigma^2 m}{N n_b} \nonumber \\
& \quad \leq \left( \mathbb{E} f \left( \bar{\mathbf{x}}_{0}^{s+1} \right) - \mathbb{E} f \left( \bar{\mathbf{x}}_{m}^{s+1} \right) \right) \label{eq_favg_online_in_epoch2} \\
\Rightarrow & \frac{1}{T} \sum_{s=0}^{S-1} \sum_{t=0}^{m-1} \mathbb{E} \big\| \nabla f \left( \bar{\mathbf{x}}_{t}^{s+1} \right) \big\|^2 + \left( 1 - L \gamma - \left[ \frac{128 \gamma^{4} L^4 m}{N B \delta^2} + \frac{4 \gamma^{2} L^2 m}{N B} + \frac{512 \gamma^{4} L^4}{\delta^4} \right] \right) \frac{1}{T} \sum_{s=0}^{S-1}  \sum_{t=0}^{m-1} \mathbb{E} \big\| \bar{\mathbf{v}}_{t}^{s+1} \big\|^2 \nonumber \\
& \quad \leq \frac{2}{T \gamma} \sum_{s=0}^{S-1} \left( \mathbb{E} f \left( \bar{\mathbf{x}}_{0}^{s+1} \right) - \mathbb{E} f \left( \bar{\mathbf{x}}_{m}^{s+1} \right) \right) + \frac{2}{T \gamma} \sum_{s=0}^{S-1} \frac{\gamma \sigma^2 m}{N n_b} \nonumber \\
& \quad = \frac{2}{T \gamma} \left( \mathbb{E} f \left( \bar{\mathbf{x}}^{0} \right) - \mathbb{E} f \left( \bar{\mathbf{x}}_{m}^{S+1} \right) \right) + \frac{2 \sigma^2}{N n_b} \nonumber \\
& \quad \leq \frac{2 \left( f(\mathbf{x}^0) - f_{\ast} \right)}{T \gamma} + \frac{2 \sigma^2}{N n_b}. \label{eq_favg_online_over_all_epcohs}
\end{align}
Similar to \eqref{eq_favg_over_all_epcohs}, for small enough step size $\gamma$, \eqref{eq_coeff_avg_v} holds. \pk{Next, we present the convergence of PR-SPIDER for the online case \eqref{eq_problem_online}. } 

\begin{theorem}{(Convergence)}
\label{thm_conv_online_sum}
\label{thm}
For the finite-sum problem under Assumptions (A1)-(A2), for small enough step size $0 < \gamma < \frac{1}{8IL}$, 
\begin{align}
    \min_{s,t} \left[ \mathbb{E} \left\| \nabla f \left( \bar{\mathbf{x}}_{t}^{s+1} \right) \right\|^2 + \frac{1}{N} \sum_{i=1}^N \mathbb{E} \left\| \mathbf{x}_{i,t}^{s+1} - \bar{\mathbf{x}}_{t}^{s+1} \right\|^2 \right] \leq \frac{2 \left( f(\mathbf{x}^0) - f_{\ast} \right)}{T \gamma} + \frac{2 \sigma^2}{N n_b}. \label{eq_final_bd_online_3}
\end{align}
\end{theorem}
\begin{proof}
See Appendix \ref{app_thm_conv_online_sum}.
\end{proof}
\pk{Note that, in comparison to Theorem \ref{thm_conv_finite_sum}, in Theorem \ref{thm_conv_online_sum} we have an additional term ${2 \sigma^2}/{N n_b}$, which accounts for the variance of the gradients computed at the start of each epoch. }

\pk{Similar to the finite sample case \eqref{eq_problem_finite}, we use Theorem \ref{thm_conv_online_sum} to compute the bounds on the sample complexity (see Definition \ref{Def: Comp}) and the communication complexity (see Definition \ref{Def: Comm}) of the proposed algorithm for the online problem \eqref{eq_problem_online}. In the following, we choose the $T$ and the parameters $n_b$, $m$ and $B$ such that such that the algorithm reaches an $\epsilon$-FoS point (see Definition \ref{Def: FoS}) and at the same time the sample and the communication complexity are minimized. }

\subsection{Sample Complexity}
\label{subsec_samp_comp}
Number of iterations $T$ satisfies
\begin{align}
    & \frac{2 \left( f(\mathbf{x}^0) - f_{\ast} \right)}{T \gamma} = \epsilon \quad \Rightarrow \quad T = \frac{2 \left( f(\mathbf{x}^0) - f_{\ast} \right)}{\gamma \epsilon} = \frac{C I}{\epsilon},
\end{align}
for constants $C, I$.  And the batch size $n_b$ to compute the gradient estimators at the start of each epoch satisfy
\begin{align}
    \frac{2 \sigma^2}{N n_b} = \frac{\epsilon}{2} \quad \Rightarrow \quad n_b = \frac{4 \sigma^2}{N \epsilon}.
\end{align} Then for $m = I \sqrt{N n_b}, B = \frac{1}{I} \sqrt{\frac{n_b}{N}}$
\begin{align*}
    & N \times \left( \left \lceil \frac{T}{m} \right \rceil \cdot n_b + T \cdot B \right) \leq N \times \left( \left( \frac{C I}{m \epsilon} + 1 \right) \cdot n_b + \frac{C I}{\epsilon} \cdot B \right) \\
    & \leq N \times \left( \frac{C I}{\epsilon} \cdot \left( \mfrac{n_b}{m} + B \right) + n_b \right) \\
    &= O \left(  \frac{\sqrt{N n_b}}{\epsilon} + N n_b \right) \\
    &= O \left(  \frac{\sigma}{\epsilon^{3/2}} + \frac{\sigma^2}{\epsilon} \right).
\end{align*}

\subsection{Communication Complexity}
Since communication happens once every $I$ iterations, the communication complexity is
\begin{align*}
    \left \lceil \frac{T}{I} \right \rceil \leq 1 + \frac{C}{\epsilon} = O \left( \frac{1}{\epsilon} \right).
\end{align*}

\section{Conclusion}
\label{sec_conclude}
In this paper, we proposed a distributed variance-reduced algorithm, PR-SPIDER, for stochastic non-convex optimization. Our algorithm is a non-trivial extension of SPIDER \cite{spider18neurips}, the single-node stochastic optimization algorithm, and parallel-restarted SGD \cite{restarted_sgd19aaai, yu19icml_batchsize}. We proved convergence of our algorithm to a first-order stationary solution. The proposed approach achieves the best known communication complexity $O(\epsilon^{-1})$. In terms of IFO complexity, we have achieved the optimal rate, significantly improving the state-of-the-art, while maintaining the linear speedup achieved by existing methods. For finite-sum problems, we achieved the optimal $O(\frac{\sqrt{N n}}{\epsilon})$ overall IFO complexity. On the other hand, for online problems, we achieved the optimal $O \left( \frac{\sigma}{\epsilon^{3/2}} + \frac{\sigma^2}{\epsilon} \right)$, a massive improvement over the existing $O(\frac{\sigma^2}{\epsilon^2})$. In addition, unlike many existing approaches, our algorithm is general enough to allow non-identical distributions of data across workers.



\bibliographystyle{IEEEtran}
\bibliography{abrv,References}


\newpage
\appendix
\newpage
\section{Finite-sum case}
\subsection{Proof of Lemma \ref{lemma_finite_bd_var}}
\label{app_lemma_finite_bd_var}
\begin{proof} We have
\begin{align}
    & \mathbb{E} \left\| \bar{\mathbf{v}}_{t}^{s+1} - \mfrac{1}{N} \sum\nolimits_{i=1}^N \nabla f_i \left( \mathbf{x}_{i,t}^{s+1} \right) \right\|^2 \nonumber \\
    &= \mathbb{E} \Big\| \bar{\mathbf{v}}_{t-1}^{s+1} - \mfrac{1}{N} \sum\nolimits_{i=1}^N \nabla f_i \left( \mathbf{x}_{i,t-1}^{s+1} \right) + \mfrac{1}{N} \sum\nolimits_{j=1}^N \mfrac{1}{B} \sum\nolimits_{\xi_{j,t}^{s+1}} \left[ \nabla f_j \big( \mathbf{x}_{j,t}^{s+1}; \xi_{j, t}^{s+1} \big) - \nabla f_j \big( \mathbf{x}_{j,t-1}^{s+1}; \xi_{j, t}^{s+1} \big) \right] \nonumber \\
    & \qquad \qquad + \mfrac{1}{N} \sum\nolimits_{i=1}^N \left[ \nabla f_i \left( \mathbf{x}_{i,t-1}^{s+1} \right) - \nabla f_i \left( \mathbf{x}_{i,t}^{s+1} \right) \right] \Big\|^2 \label{eq_upbd_avg_gradf_i_avgv_1a} \\
    &= \mathbb{E} \left\| \bar{\mathbf{v}}_{t-1}^{s+1} - \mfrac{1}{N} \sum\nolimits_{i=1}^N \nabla f_i \left( \mathbf{x}_{i,t-1}^{s+1} \right) \right\|^2 \nonumber \\
    & + \mathbb{E} \Big\| \mfrac{1}{N} \sum_{i=1}^N \mfrac{1}{B} \sum_{\xi_{i,t}^{s+1}} \left[ \nabla f_i \big( \mathbf{x}_{i,t}^{s+1}; \xi_{i, t}^{s+1} \big) - \nabla f_i \big( \mathbf{x}_{i,t-1}^{s+1}; \xi_{i, t}^{s+1} \big) + \nabla f_i \left( \mathbf{x}_{i,t-1}^{s+1} \right) - \nabla f_i \left( \mathbf{x}_{i,t}^{s+1} \right) \right] \Big\|^2 \nonumber \\
    & + \mathbb{E} \left\langle \bar{\mathbf{v}}_{t-1}^{s+1} - \mfrac{1}{N} \sum_{i=1}^N \nabla f_i \left( \mathbf{x}_{i,t-1}^{s+1} \right), \underbrace{\mfrac{1}{N} \sum_{i=1}^N \mfrac{1}{B} \sum_{\xi_{i,t}^{s+1}} \left[ \nabla f_i \big( \mathbf{x}_{i,t}^{s+1}; \xi_{i, t}^{s+1} \big) - \nabla f_i \big( \mathbf{x}_{i,t-1}^{s+1}; \xi_{i, t}^{s+1} \big) + \nabla f_i \left( \mathbf{x}_{i,t-1}^{s+1} \right) - \nabla f_i \left( \mathbf{x}_{i,t}^{s+1} \right) \right]}_{\mathbb{E}(\cdot) = \mathbf{0}} \right\rangle \label{eq_upbd_avg_gradf_i_avgv_1} \\
    &= \mathbb{E} \left\| \bar{\mathbf{v}}_{t-1}^{s+1} - \mfrac{1}{N} \sum\nolimits_{i=1}^N \nabla f_i \left( \mathbf{x}_{i,t-1}^{s+1} \right) \right\|^2 \nonumber \\
    & \quad + \mfrac{1}{N^2} \sum\nolimits_{i=1}^N \mathbb{E} \Big\| \mfrac{1}{B} \sum\nolimits_{\xi_{i,t}^{s+1}} \left[ \nabla f_i \big( \mathbf{x}_{i,t}^{s+1}; \xi_{i, t}^{s+1} \big) - \nabla f_i \big( \mathbf{x}_{i,t-1}^{s+1}; \xi_{i, t}^{s+1} \big) + \nabla f_i \left( \mathbf{x}_{i,t-1}^{s+1} \right) - \nabla f_i \left( \mathbf{x}_{i,t}^{s+1} \right) \right] \Big\|^2 \nonumber \\
    & \quad + \mfrac{1}{N^2} \sum\nolimits_{i \neq j} \mathbb{E} \left\langle \mfrac{1}{B} \sum\nolimits_{\xi_{i,t}^{s+1}} \left[ \nabla f_i \big( \mathbf{x}_{i,t}^{s+1}; \xi_{i, t}^{s+1} \big) - \nabla f_i \big( \mathbf{x}_{i,t-1}^{s+1}; \xi_{i, t}^{s+1} \big) + \nabla f_i \left( \mathbf{x}_{i,t-1}^{s+1} \right) - \nabla f_i \left( \mathbf{x}_{i,t}^{s+1} \right) \right], \right. \nonumber \\
    & \qquad \qquad \qquad \qquad \left. \mfrac{1}{B} \sum\nolimits_{\xi_{j,t}^{s+1}} \left[ \nabla f_j \big( \mathbf{x}_{j,t}^{s+1}; \xi_{j, t}^{s+1} \big) - \nabla f_j \big( \mathbf{x}_{j,t-1}^{s+1}; \xi_{j, t}^{s+1} \big) + \nabla f_j \left( \mathbf{x}_{j,t-1}^{s+1} \right) - \nabla f_j \left( \mathbf{x}_{j,t}^{s+1} \right) \right] \right\rangle \label{eq_upbd_avg_gradf_i_avgv_2} \\
    &= \mathbb{E} \left\| \bar{\mathbf{v}}_{t-1}^{s+1} - \mfrac{1}{N} \sum\nolimits_{i=1}^N \nabla f_i \left( \mathbf{x}_{i,t-1}^{s+1} \right) \right\|^2 \nonumber \\
    & \quad + \mfrac{1}{N^2} \sum\nolimits_{i=1}^N \mfrac{1}{B^2} \mathbb{E} \sum\nolimits_{\xi_{i,t}^{s+1}} \Big\| \nabla f_i \big( \mathbf{x}_{i,t}^{s+1}; \xi_{i, t}^{s+1} \big) - \nabla f_i \big( \mathbf{x}_{i,t-1}^{s+1}; \xi_{i, t}^{s+1} \big) + \nabla f_i \left( \mathbf{x}_{i,t-1}^{s+1} \right) - \nabla f_i \left( \mathbf{x}_{i,t}^{s+1} \right) \Big\|^2 \nonumber \\
    & \quad + \mfrac{1}{N^2} \sum\nolimits_{i=1}^N \frac{1}{B^2}  \mathbb{E} \sum\nolimits_{\xi_{i,t}^{s+1} \neq \zeta_{i,t}^{s+1}}  \left\langle \underbrace{\mathbb{E}_{\xi_{i,t}^{s+1}} \left[ \nabla f_i \big( \mathbf{x}_{i,t}^{s+1}; \xi_{i, t}^{s+1} \big) - \nabla f_i \big( \mathbf{x}_{i,t-1}^{s+1}; \xi_{i, t}^{s+1} \big) \right] + \nabla f_i \left( \mathbf{x}_{i,t-1}^{s+1} \right) - \nabla f_i \left( \mathbf{x}_{i,t}^{s+1} \right)}_{= \mathbf{0}}, \right. \nonumber \\
    & \qquad \qquad \qquad \qquad \qquad \qquad \qquad \qquad \left. \underbrace{\mathbb{E}_{\zeta_{i,t}^{s+1}} \left[ \nabla f_i \big( \mathbf{x}_{i,t}^{s+1}; \zeta_{i, t}^{s+1} \big) - \nabla f_i \big( \mathbf{x}_{i,t-1}^{s+1}; \zeta_{i, t}^{s+1} \big) \right] + \nabla f_i \left( \mathbf{x}_{i,t-1}^{s+1} \right) - \nabla f_i \left( \mathbf{x}_{i,t}^{s+1} \right)}_{= \mathbf{0}} \right\rangle \label{eq_upbd_avg_gradf_i_avgv_3} \\
    & \leq \mathbb{E} \left\| \bar{\mathbf{v}}_{t-1}^{s+1} - \mfrac{1}{N} \sum\nolimits_{i=1}^N \nabla f_i \left( \mathbf{x}_{i,t-1}^{s+1} \right) \right\|^2 + \frac{1}{N^2} \sum\nolimits_{i=1}^N \frac{1}{B^2} \mathbb{E} \sum\nolimits_{\xi_{i,t}^{s+1}}  \Big\| \nabla f_i \big( \mathbf{x}_{i,t}^{s+1}; \xi_{i, t}^{s+1} \big) - \nabla f_i \big( \mathbf{x}_{i,t-1}^{s+1}; \xi_{i, t}^{s+1} \big)\Big\|^2 \label{eq_upbd_avg_gradf_i_avgv_4} \\
    & \leq \mathbb{E} \left\| \bar{\mathbf{v}}_{t-1}^{s+1} - \mfrac{1}{N} \sum\nolimits_{i=1}^N \nabla f_i \left( \mathbf{x}_{i,t-1}^{s+1} \right) \right\|^2 + \frac{1}{N^2} \sum\nolimits_{i=1}^N \frac{L^2}{B} \mathbb{E} \Big\| \mathbf{x}_{i,t}^{s+1} - \mathbf{x}_{i,t-1}^{s+1} \Big\|^2 \label{eq_upbd_avg_gradf_i_avgv_5} \\
    & \leq \underbrace{\mathbb{E} \left\| \bar{\mathbf{v}}_{0}^{s+1} - \mfrac{1}{N} \sum\nolimits_{i=1}^N \nabla f_i \left( \mathbf{x}_{i,0}^{s+1} \right) \right\|^2}_{E_0^{s+1}} + \frac{L^2}{N^2 B} \sum_{i=1}^N \sum_{\ell = 0}^{t-1} \mathbb{E} \Big\| \mathbf{x}_{i,\ell + 1}^{s+1} - \mathbf{x}_{i,\ell}^{s+1} \Big\|^2 \label{eq_upbd_avg_gradf_i_avgv_6}
\end{align}
\eqref{eq_upbd_avg_gradf_i_avgv_1a} follows from step \ref{alg1_inner_loop_dir} in Algorithm \ref{alg1}. The cross term in \eqref{eq_upbd_avg_gradf_i_avgv_1} is zero since
\begin{align}
    \mathbb{E}_{\xi_{i, t}^{s+1}} \left[ \nabla f_i \big( \mathbf{x}_{i,t}^{s+1}; \xi_{i, t}^{s+1} \big) - \nabla f_i \big( \mathbf{x}_{i,t-1}^{s+1}; \xi_{i, t}^{s+1} \big) \right] = \nabla f_i \left( \mathbf{x}_{i,t}^{s+1} \right) - \nabla f_i \left( \mathbf{x}_{i,t-1}^{s+1} \right), \label{eq_stoch_grad_unbias}
\end{align}
for all $\xi_{i, t}^{s+1} \in \mathcal{B}_{i,t}^{s+1}$, $\forall \ t, \forall \ s, \forall \ i \in [1:N]$. The cross term in \eqref{eq_upbd_avg_gradf_i_avgv_2} is zero since the mini-batches $\mathcal{B}_{i,t}^{s+1}$ are sampled uniformly randomly, and independently at all the nodes $i \in [1:N]$. The cross term in \eqref{eq_upbd_avg_gradf_i_avgv_3} is zero since at a single node $i$, the samples in the mini-batch $\mathcal{B}_{i,t}^{s+1}$ are sampled independent of each other. \eqref{eq_upbd_avg_gradf_i_avgv_4} follows from \eqref{eq_stoch_grad_unbias} and using $\mathbb{E} \left\| \mathbf{x} - \mathbb{E} (\mathbf{x}) \right\|^2 \leq \mathbb{E} \left\| \mathbf{x} \right\|^2$. \eqref{eq_upbd_avg_gradf_i_avgv_5} follows from mean-squared L-smooth property of each stochastic function $f_i(\cdot, \xi)$. \eqref{eq_upbd_avg_gradf_i_avgv_6} follows by recursive application of \eqref{eq_upbd_avg_gradf_i_avgv_5}, to the beginning of the epoch.
\end{proof}

\subsection{Proof of Lemma \ref{lemma_finite_nw_error}}
\label{app_lemma_finite_nw_error}
\begin{proof}
First we prove \eqref{eq_lemma_finite_nw_error_1}. For $\ell$ such that $\ell \text{ mod } I \neq 0$
\begin{align}
& \sum_{i=1}^N \mathbb{E} \left\| \mathbf{v}_{i,\ell}^{s+1} - \bar{\mathbf{v}}_{\ell}^{s+1}  \right\|^2 \leq \sum_{i=1}^N \Bigg[ \left( 1 + \delta \right) \mathbb{E} \left\| \mathbf{v}_{i,\ell-1}^{s+1} - \bar{\mathbf{v}}_{\ell-1}^{s+1}  \right\|^2 + \left( 1 + \mfrac{1}{\delta} \right) \mathbb{E} \Big\| \mfrac{1}{B} \sum_{\xi_{i,\ell}^{s+1}} \left\{ \nabla f_i \big( \mathbf{x}_{i,\ell}^{s+1}; \xi_{i, \ell}^{s+1} \big) - \nabla f_i \big( \mathbf{x}_{i,\ell-1}^{s+1}; \xi_{i, \ell}^{s+1} \big) \right\} \nonumber \\
& \qquad \qquad \qquad \qquad \qquad - \mfrac{1}{N} \sum_{j=1}^N \mfrac{1}{B} \sum_{\xi_{j,\ell}^{s+1}} \left\{ \nabla f_j \big( \mathbf{x}_{j,\ell}^{s+1}; \xi_{j, \ell}^{s+1} \big) - \nabla f_j \big( \mathbf{x}_{j,\ell-1}^{s+1}; \xi_{j, \ell}^{s+1} \big) \right\} \Big\|^2 \Bigg] \label{eq_upbd_v_avgv_1} \\
& \leq \sum_{i=1}^N \Bigg[ \left( 1 + \delta \right) \mathbb{E} \left\| \mathbf{v}_{i,\ell-1}^{s+1} - \bar{\mathbf{v}}_{\ell-1}^{s+1} \right\|^2 + \left( 1 + \mfrac{1}{\delta} \right) 2 \mathbb{E} \Big\| \frac{1}{B} \sum_{\xi_{i,\ell}^{s+1}} \left\{ \nabla f_i \big( \mathbf{x}_{i,\ell}^{s+1}; \xi_{i, \ell}^{s+1} \big) - \nabla f_i \big( \mathbf{x}_{i,\ell-1}^{s+1}; \xi_{i, \ell}^{s+1} \big) \right\} \Big\|^2 \nonumber \\
& \qquad \qquad + \left( 1 + \mfrac{1}{\delta} \right) 2 \mathbb{E} \Big\| \frac{1}{N} \sum_{j=1}^N \frac{1}{B} \sum_{\xi_{j,\ell}^{s+1}} \left\{ \nabla f_j \big( \mathbf{x}_{j,\ell}^{s+1}; \xi_{j, \ell}^{s+1} \big) - \nabla f_j \big( \mathbf{x}_{j,\ell-1}^{s+1}; \xi_{j, \ell}^{s+1} \big) \right\} \Big\|^2 \Bigg] \nonumber \\
& \leq \sum_{i=1}^N \Bigg[ \left( 1 + \delta \right) \mathbb{E} \left\| \mathbf{v}_{i,\ell-1}^{s+1} - \bar{\mathbf{v}}_{\ell-1}^{s+1} \right\|^2 + \left( 1 + \mfrac{1}{\delta} \right) \mfrac{2}{B} \mathbb{E}  \sum_{\xi_{i,\ell}^{s+1}} \Big\| \nabla f_i \big( \mathbf{x}_{i,\ell}^{s+1}; \xi_{i, \ell}^{s+1} \big) - \nabla f_i \big( \mathbf{x}_{i,\ell-1}^{s+1}; \xi_{i, \ell}^{s+1} \big) \Big\|^2 \nonumber \\
& \qquad \qquad + \left( 1 + \mfrac{1}{\delta} \right) \mfrac{2}{N} \sum_{j=1}^N \mfrac{1}{B} \mathbb{E} \sum_{\xi_{j,\ell}^{s+1}}  \Big\| \nabla f_j \big( \mathbf{x}_{j,\ell}^{s+1}; \xi_{j, \ell}^{s+1} \big) - \nabla f_j \big( \mathbf{x}_{j,\ell-1}^{s+1}; \xi_{j, \ell}^{s+1} \big) \Big\|^2 \Bigg] \label{eq_upbd_v_avgv_2} \\
& \leq \sum_{i=1}^N \Big[ \left( 1 + \delta \right) \mathbb{E} \left\| \mathbf{v}_{i,\ell-1}^{s+1} - \bar{\mathbf{v}}_{\ell-1}^{s+1} \right\|^2 + \left( 1 + \mfrac{1}{\delta} \right) 2 L^2 \mathbb{E} \left\| \mathbf{x}_{i,\ell}^{s+1} - \mathbf{x}_{i,\ell-1}^{s+1} \right\|^2 + \left( 1 + \mfrac{1}{\delta} \right) \mfrac{2 L^2}{N} \sum_{j=1}^N \mathbb{E} \left\| \mathbf{x}_{j,\ell}^{s+1} - \mathbf{x}_{j,\ell-1}^{s+1} \right\|^2 \Big] \label{eq_upbd_v_avgv_3} \\
&= \sum_{i=1}^N \Big[ \left( 1 + \delta \right) \mathbb{E} \left\| \mathbf{v}_{i,\ell-1}^{s+1} - \bar{\mathbf{v}}_{\ell-1}^{s+1} \right\|^2 + 4 L^2 \left( 1 + \mfrac{1}{\delta} \right) \mathbb{E} \left\| \mathbf{x}_{i,\ell}^{s+1} - \mathbf{x}_{i,\ell-1}^{s+1} \right\|^2 \Big] \nonumber \\
& \leq \sum_{i=1}^N \Big[ \left( 1 + \delta \right) \mathbb{E} \left\| \mathbf{v}_{i,\ell-1}^{s+1} - \bar{\mathbf{v}}_{\ell-1}^{s+1} \right\|^2 + 8 \gamma^2 L^2 \left( 1 + \mfrac{1}{\delta} \right) \mathbb{E} \left\| \mathbf{v}_{i,\ell-1}^{s+1} - \bar{\mathbf{v}}_{\ell-1}^{s+1}  \right\|^2 + 8 \gamma^2 L^2 \left( 1 + \mfrac{1}{\delta} \right) \mathbb{E} \left\| \bar{\mathbf{v}}_{\ell-1}^{s+1} \right\|^2 \Big] \nonumber \\
& \leq \sum_{i=1}^N \Big( 1 + \underbrace{\delta + 8 \gamma^2 L^2 \big( 1 + \mfrac{1}{\delta} \big)}_{\theta} \Big) \mathbb{E} \left\| \mathbf{v}_{i,\ell-1}^{s+1} - \bar{\mathbf{v}}_{\ell-1}^{s+1} \right\|^2 + 8 \gamma^2 N L^2 \big( 1 + \mfrac{1}{\delta} \big) \mathbb{E} \left\| \bar{\mathbf{v}}_{\ell-1}^{s+1}  \right\|^2 \label{eq_upbd_v_avgv_4}
\end{align}
where, \eqref{eq_upbd_v_avgv_1} follows from step \ref{alg1_inner_loop_dir} in Algorithm \ref{alg1} and Young's inequality, for $\delta > 0$; \eqref{eq_upbd_v_avgv_2} follows from Jensen's inequality; \eqref{eq_upbd_v_avgv_3} follows from the mean-squared $L$-smooth assumption (A1). Applying \eqref{eq_upbd_v_avgv_4} recursively, we obtain

\begin{align}
    & \sum_{i=1}^N \mathbb{E} \left\| \mathbf{v}_{i,\ell}^{s+1} - \bar{\mathbf{v}}_{\ell}^{s+1}  \right\|^2 \leq \sum_{i=1}^N (1 + \theta)^2 \mathbb{E} \left\| \mathbf{v}_{i,\ell-2}^{s+1} - \bar{\mathbf{v}}_{\ell-2}^{s+1}  \right\|^2 + 8 \gamma^2 N L^2 \big( 1 + \mfrac{1}{\delta} \big) \Big[ \mathbb{E} \left\| \bar{\mathbf{v}}_{\ell-1}^{s+1}  \right\|^2 + (1+\theta) \mathbb{E} \left\| \bar{\mathbf{v}}_{\ell-2}^{s+1}  \right\|^2 \Big] \nonumber \\
    & \leq \sum_{i=1}^N (1 + \theta)^{\ell - \tau(\ell)} \mathbb{E} \left\| \mathbf{v}_{i,\tau(\ell)}^{s+1} - \bar{\mathbf{v}}_{\tau(\ell)}^{s+1}  \right\|^2 + 8 \gamma^2 N L^2 \big( 1 + \mfrac{1}{\delta} \big) \sum_{j=\tau(\ell)}^{\ell-1} (1+\theta)^{\ell - 1 - j} \mathbb{E} \left\| \bar{\mathbf{v}}_{j}^{s+1} \right\|^2 \nonumber \\
    &= 8 \gamma^2 N L^2 \big( 1 + \mfrac{1}{\delta} \big) \sum_{j=\tau(\ell)}^{\ell-1} (1+\theta)^{\ell-1-j} \mathbb{E} \left\| \bar{\mathbf{v}}_{j}^{s+1}  \right\|^2\label{eq_upbd_v_avgv_5}
\end{align}
where \eqref{eq_upbd_v_avgv_5} follows since averaging happens at time index $\tau(\ell)$ (step \ref{alg1_v_avg_inloop}, Algorithm \ref{alg1}), i.e., $\mathbf{v}_{i,\tau(\ell)}^{s+1} = \bar{\mathbf{v}}_{\tau(\ell)}^{s+1}$, $\forall \ i \in [1:N]$.

Next, we prove \eqref{eq_lemma_finite_nw_error_2}. For $\ell$ such that $\ell \text{ mod } I \neq 0$
\begin{align}
& \sum_{i=1}^N \mathbb{E} \left\| \mathbf{x}_{i,\ell}^{s+1} - \bar{\mathbf{x}}_{\ell}^{s+1} \right\|^2 = \sum_{i=1}^N \mathbb{E} \left\| \left[ \mathbf{x}_{i,\ell-1}^{s+1} - \gamma \mathbf{v}_{i,\ell-1}^{s+1} \right] - \left[ \bar{\mathbf{x}}_{\ell-1}^{s+1}  -\gamma \bar{\mathbf{v}}_{t-1}^{s+1}  \right] \right\|^2 \nonumber \\
&  \quad \leq \sum_{i=1}^N \left[ (1 + \alpha) \mathbb{E} \left\| \mathbf{x}_{i,\ell-1}^{s+1} - \bar{\mathbf{x}}_{\ell-1}^{s+1}  \right\|^2 + \left( 1 + \mfrac{1}{\alpha} \right) \gamma^2 \mathbb{E} \left\| \mathbf{v}_{i,\ell-1}^{s+1} - \bar{\mathbf{v}}_{\ell-1}^{s+1}  \right\|^2 \right] \label{eq_upbd_x_avgx_1} \\
& \quad \leq \sum_{i=1}^N \left[ (1 + \alpha)^2 \mathbb{E} \left\| \mathbf{x}_{i,\ell-2}^{s+1} - \bar{\mathbf{x}}_{\ell-2}^{s+1}  \right\|^2 + \left( 1 + \mfrac{1}{\alpha} \right) \gamma^2 \left\{ \mathbb{E} \left\| \mathbf{v}_{i,\ell-1}^{s+1} - \bar{\mathbf{v}}_{\ell-1}^{s+1}  \right\|^2 + (1+\alpha) \mathbb{E} \left\| \mathbf{v}_{i,\ell-2}^{s+1} - \bar{\mathbf{v}}_{\ell-2}^{s+1}  \right\|^2 \right\} \right] \nonumber \\
& \quad \leq \sum_{i=1}^N \left[ (1 + \alpha)^{\ell-\tau(\ell)} \mathbb{E} \left\| \mathbf{x}_{i,\tau(\ell)}^{s+1} - \bar{\mathbf{x}}_{\tau(\ell)}^{s+1} \right\|^2 + \left( 1 + \mfrac{1}{\alpha} \right) \gamma^2 \sum_{j=\tau(\ell)}^{\ell - 1} (1+\alpha)^{\ell - 1 - j} \mathbb{E} \left\| \mathbf{v}_{i,j}^{s+1} - \bar{\mathbf{v}}_{j}^{s+1} \right\|^2  \right] \nonumber \\
& \quad = \left( 1 + \mfrac{1}{\alpha} \right) \gamma^2 \sum_{i=1}^N \sum_{j=\tau(\ell)+1}^{\ell-1} (1+\alpha)^{\ell-1-j} \mathbb{E} \left\| \mathbf{v}_{i,j}^{s+1} - \bar{\mathbf{v}}_{j}^{s+1}  \right\|^2 \label{eq_upbd_x_avgx_2}
\end{align}
where \eqref{eq_upbd_x_avgx_1} follows from Young's inequality, with $\alpha>0$; \eqref{eq_upbd_x_avgx_2} follows since averaging happens at time index $\tau(\ell)$ (step \ref{alg1_x_avg_inloop}-\ref{alg1_v_avg_inloop}, Algorithm \ref{alg1}). Consequently $\mathbf{x}_{i,\tau(\ell)}^{s+1} = \bar{\mathbf{x}}_{\tau(\ell)}^{s+1}, \mathbf{v}_{i,\tau(\ell)}^{s+1} = \bar{\mathbf{v}}_{\tau(\ell)}^{s+1}$, $\forall \ i \in [1:N]$. We can further upper bound \eqref{eq_upbd_x_avgx_2} using the bound on $\sum_{i=1}^N \mathbb{E} \| \mathbf{v}_{i,j}^{s+1} - \bar{\mathbf{v}}_{j}^{s+1} \|^2$ in \eqref{eq_upbd_v_avgv_5}.
\end{proof}

\subsection{Proof of Lemma \ref{lemma_finite_descent}}
\label{app_lemma_finite_descent}
\pk{To prove Lemma \ref{lemma_finite_descent}, we need some preliminary results which are given before providing the proof of Lemma \ref{lemma_finite_descent}.   
\subsubsection{Preliminary results required for Lemma \ref{lemma_finite_descent}}
In the following analysis, the following facts shall be utilized repeatedly.
\begin{itemize}
    \item[(F1)] We choose $\delta$, $\gamma$ such that $\delta < \theta = \delta + 8 \gamma^2 L^2 \big( 1 + \frac{1}{\delta} \big) < 2 \delta < 1$.
    \item[(F2)] We choose $\alpha = \frac{\delta}{2}$. Therefore, $\theta - \alpha > \frac{\delta}{2}$.
    \item[(F3)] $\big( 1 + \frac{1}{\delta} \big) < \frac{2}{\delta}$ and $\big( 1 + \frac{1}{\alpha} \big) < \frac{2}{\alpha}$.
    \item[(F4)] $I = \left \lfloor \frac{1}{\theta} \right \rfloor \leq \frac{1}{\theta}$. Also, for $\theta < 1, \left \lfloor \frac{1}{\theta} \right \rfloor \geq \frac{1}{2 \theta} > \frac{1}{4\delta}$.
\end{itemize}
Now, we state first of the three lemmas we will require to prove Lemma \ref{lemma_finite_descent}.
\begin{lemma}
\label{lem_finite_descent1}
We have:
$$\frac{16 \gamma^5 L^4}{N B} \big( 1 + \mfrac{1}{\delta} \big) \sum_{t=0}^{m-1} \sum_{\ell=0}^{t-1} \sum_{j=\tau(\ell)}^{\ell-1} (1+\theta)^{\ell-1-j} \mathbb{E} \Big\| \bar{\mathbf{v}}_{j}^{s+1}  \Big\|^2 \leq \frac{64 \gamma^5 L^4 m}{N B \delta^2} \sum_{t=0}^{m-1} \mathbb{E} \left\| \bar{\mathbf{v}}_{t}^{s+1} \right\|^2.$$
\end{lemma}
\begin{proof}
From the left hand side of the inequality, we have
\begin{align} 
        & \frac{16 \gamma^5 L^4}{N B} \big( 1 + \mfrac{1}{\delta} \big) \sum_{t=0}^{m-1} \sum_{\ell=0}^{t-1} \sum_{j=\tau(\ell)}^{\ell-1} (1+\theta)^{\ell-1-j} \mathbb{E} \Big\| \bar{\mathbf{v}}_{j}^{s+1}  \Big\|^2 \nonumber \\
        & \ = \frac{16 \gamma^5 L^4}{N B} \big( 1 + \mfrac{1}{\delta} \big) \sum_{t=0}^{m-1} \left[ \sum_{\ell=0}^{I-1} \sum_{j=0}^{\ell-1} (1+\theta)^{\ell-1-j} \mathbb{E} \left\| \bar{\mathbf{v}}_{j}^{s+1}  \right\|^2 \right. \nonumber \\
        & \qquad \left. + \sum_{\ell=I}^{2I-1} \sum_{j=I}^{\ell-1} (1+\theta)^{\ell-1-j} \mathbb{E} \left\| \bar{\mathbf{v}}_{j}^{s+1}  \right\|^2 \hdots + \sum_{\ell=\tau(t)}^{t-1} \sum_{j=\tau(t)}^{\ell-1} (1+\theta)^{\ell-1-j} \mathbb{E} \left\| \bar{\mathbf{v}}_{j}^{s+1}  \right\|^2 \right] \label{eq_phim_V1_1} \\
        & \ = \frac{16 \gamma^5 L^4}{N B} \big( 1 + \mfrac{1}{\delta} \big) \sum_{t=0}^{m-1} \left[ \sum_{\ell=0}^{I-1} \sum_{j=0}^{\ell-1} (1+\theta)^{\ell-1-j} \mathbb{E} \left\| \bar{\mathbf{v}}_{j}^{s+1}  \right\|^2 \right. \nonumber \\
        & \qquad \left. + \sum_{\ell=0}^{I-1} \sum_{j=0}^{\ell-1} (1+\theta)^{\ell-1-j} \mathbb{E} \left\| \bar{\mathbf{v}}_{j+I}^{s+1}  \right\|^2 \hdots + \sum_{\ell=0}^{t-1-\tau(t)} \sum_{j=0}^{\ell-1} (1+\theta)^{\ell-1-j} \mathbb{E} \left\| \bar{\mathbf{v}}_{j+\tau(t)}^{s+1}  \right\|^2 \right] \label{eq_phim_V1_1b} \\
        & \ \leq \frac{16 \gamma^5 L^4}{N B} \big( 1 + \mfrac{1}{\delta} \big) \left[  \sum_{t=0}^{m-1} \sum_{\ell=0}^{I-1} (1+\theta)^{\ell-1} \right] \sum_{j=0}^{m-1} \mathbb{E} \left\| \bar{\mathbf{v}}_{j}^{s+1}  \right\|^2 \label{eq_phim_V1_2} \\
        & \ \leq \frac{16 \gamma^5 L^4}{N B} \big( 1 + \mfrac{1}{\delta} \big) \sum_{t=0}^{m-1} \mathbb{E} \left\| \bar{\mathbf{v}}_{t}^{s+1} \right\|^2 m \left[ \frac{(1+\theta)^{I} - 1}{\theta} \right] \nonumber \\
        & \ \leq \frac{16 \gamma^5 L^4 m}{N B} \frac{2}{\delta} \frac{e - 1}{\theta} \sum_{t=0}^{m-1} \mathbb{E} \left\| \bar{\mathbf{v}}_{t}^{s+1} \right\|^2 \nonumber \\
        & \ \overset{(c)}{\leq} \frac{64 \gamma^5 L^4 m}{N B \delta^2} \sum_{t=0}^{m-1} \mathbb{E} \left\| \bar{\mathbf{v}}_{t}^{s+1} \right\|^2 \label{eq_phim_V1}
    \end{align}
    where, \eqref{eq_phim_V1_1b} follows from \eqref{eq_phim_V1_1} by simple re-indexing. \eqref{eq_phim_V1_2} follows from \eqref{eq_phim_V1_1b} by upper bounding the coefficients of all terms with the one corresponding to $j=0$. Note that $(1+\frac{1}{x})^x$ in increasing function for $x>0$ and $\lim_{x \to \infty} (1+\frac{1}{x})^x = e$ (Euler's constant). \eqref{eq_phim_V1} follows from (F1)-(F3).
    \end{proof}
    Next, we present the second lemma we will use to prove Lemma \ref{lemma_finite_descent}.}
    
    \pk{\begin{lemma}
    \label{lem_finite_descent2}
    We have:
     \begin{align*}
     \frac{2 \gamma^3 L^2}{N B} \sum_{t=0}^{m-1} \sum_{\ell = 0}^{t-1} \mathbb{E} \left\| \bar{\mathbf{v}}_{\ell}^{s+1}  \right\|^2  \leq \frac{2 \gamma^3 L^2 m}{N B} \sum_{t=0}^{m-1} \mathbb{E} \left\| \bar{\mathbf{v}}_{t}^{s+1} \right\|^2.  
    \end{align*}
    \end{lemma}
    \begin{proof} We have from the left hand side of the inequality
    \begin{align*}
        & \frac{2 \gamma^3 L^2}{N B} \sum_{t=0}^{m-1} \sum_{\ell = 0}^{t-1} \mathbb{E} \left\| \bar{\mathbf{v}}_{\ell}^{s+1}  \right\|^2 \nonumber \\
        & \ = \frac{2 \gamma^3 L^2}{N B} \left[ \mathbb{E} \left\| \bar{\mathbf{v}}_{0}^{s+1}  \right\|^2 (m-1) + \mathbb{E} \left\| \bar{\mathbf{v}}_{1}^{s+1} \right\|^2 (m-2) + \hdots + \mathbb{E} \left\| \bar{\mathbf{v}}_{m-2}^{s+1} \right\|^2 (1) \right] \nonumber \\
        & \ \leq \frac{2 \gamma^3 L^2 m}{N B} \sum_{t=0}^{m-1} \mathbb{E} \left\| \bar{\mathbf{v}}_{t}^{s+1} \right\|^2.  
    \end{align*}
    \end{proof}
    Finally, we present the third intermediate lemma before presenting the proof of Lemma \ref{lemma_finite_descent}.}
    
    \pk{\begin{lemma}
    \label{lem_finite_descent3}
    We have:
    $$ 8 \gamma^5 L^4 \left( 1 + \mfrac{1}{\alpha} \right) \big( 1 + \mfrac{1}{\delta} \big) \sum_{t=0}^{m-1} \sum_{\ell=\tau(t)+1}^{t-1} \sum_{j=\tau(\ell)}^{\ell-1} (1+\alpha)^{t-1-\ell} (1+\theta)^{\ell-1-j} \mathbb{E} \left\| \bar{\mathbf{v}}_{j}^{s+1}  \right\|^2 \leq \frac{256 \gamma^5 L^4}{\delta^4} \sum_{t=0}^{m-1} \mathbb{E} \left\| \bar{\mathbf{v}}_{t}^{s+1} \right\|^2.$$
    \end{lemma}
    \begin{proof}
    Suppose $\tau(m-1) = (p-1)I$ (see \eqref{eq_tau}), we have
    \begin{align}
        & 8 \gamma^5 L^4 \left( 1 + \mfrac{1}{\alpha} \right) \big( 1 + \mfrac{1}{\delta} \big) \sum_{t=0}^{m-1} \sum_{\ell=\tau(t)+1}^{t-1} \sum_{j=\tau(\ell)}^{\ell-1} (1+\alpha)^{t-1-\ell} (1+\theta)^{\ell-1-j} \mathbb{E} \left\| \bar{\mathbf{v}}_{j}^{s+1}  \right\|^2 \nonumber \\
        & \ = 8 \gamma^5 L^4 \left( 1 + \mfrac{1}{\alpha} \right) \left( 1 + \mfrac{1}{\delta} \right) \left[ \sum_{t=0}^{I-1} \sum_{\ell=1}^{t-1} \sum_{j=0}^{\ell-1} (1+\alpha)^{t-1-\ell} (1+\theta)^{\ell-1-j} \mathbb{E} \left\| \bar{\mathbf{v}}_{j}^{s+1}  \right\|^2 \right. \nonumber \\
        & \qquad \qquad \left. + \sum_{t=I}^{2I-1} \sum_{\ell=I+1}^{t-1} \sum_{j=I}^{\ell-1} (1+\alpha)^{t-1-\ell} (1+\theta)^{\ell-1-j} \mathbb{E} \left\| \bar{\mathbf{v}}_{j}^{s+1}  \right\|^2 \hdots \right. \nonumber \\
        & \qquad \qquad \left. + \sum_{t=(p-1)I}^{m-1} \sum_{\ell=(p-1)I+1}^{t-1} \sum_{j=(p-1)I}^{\ell-1} (1+\alpha)^{t-1-\ell} (1+\theta)^{\ell-1-j} \mathbb{E} \left\| \bar{\mathbf{v}}_{j}^{s+1}  \right\|^2 \right] \label{eq_phim_V3_1} \\
        & \ = 8 \gamma^5 L^4 \left( 1 + \mfrac{1}{\alpha} \right) \left( 1 + \mfrac{1}{\delta} \right) \left[ \sum_{t=0}^{I-1} \sum_{\ell=1}^{t-1} \sum_{j=0}^{\ell-1} (1+\alpha)^{t-1-\ell} (1+\theta)^{\ell-1-j} \mathbb{E} \left\| \bar{\mathbf{v}}_{j}^{s+1}  \right\|^2 \right. \nonumber \\
        & \qquad \qquad \left. + \sum_{t=0}^{I-1} \sum_{\ell=1}^{t-1} \sum_{j=0}^{\ell-1} (1+\alpha)^{t-1-\ell} (1+\theta)^{\ell-1-j} \mathbb{E} \left\| \bar{\mathbf{v}}_{j+I}^{s+1}  \right\|^2 \hdots \right. \nonumber \\
        & \qquad \qquad \left. + \sum_{t=0}^{m-1-(p-1)I} \sum_{\ell=1}^{t-1} \sum_{j=0}^{\ell-1} (1+\alpha)^{t-1-\ell} (1+\theta)^{\ell-1-j} \mathbb{E} \left\| \bar{\mathbf{v}}_{j+(p-1)I}^{s+1}  \right\|^2 \right] \nonumber \\
        & \ \leq 8 \gamma^5 L^4 \frac{2}{\alpha} \frac{2}{\delta} \sum_{j=0}^{m-1} \mathbb{E} \left\| \bar{\mathbf{v}}_{j}^{s+1} \right\|^2 \left[ \sum_{t=0}^{I-1} \sum_{\ell=0}^{t-1} (1+\alpha)^{t-1-\ell} (1 + \theta)^{\ell-1} \right] \label{eq_phim_V3_2} \\
        & \ = \frac{32 \gamma^5 L^4}{\alpha \delta} \sum_{j=0}^{m-1} \mathbb{E} \left\| \bar{\mathbf{v}}_{j}^{s+1} \right\|^2 \left[ \sum_{t=0}^{I-1} \frac{(1+\alpha)^{t-1}}{(1+\theta)} \frac{\left( \frac{1+\theta}{1+\alpha} \right)^t - 1}{\left( \frac{1+\theta}{1+\alpha} \right) - 1} \right] \nonumber \\
        & \ \leq \frac{32 \gamma^5 L^4}{\alpha \delta} \sum_{j=0}^{m-1} \mathbb{E} \left\| \bar{\mathbf{v}}_{j}^{s+1} \right\|^2 \left[ \sum_{t=0}^{I-1}  \frac{(1+\theta)^{t} - (1+\alpha)^{t}}{\theta - \alpha} \right] \nonumber \\
        & \ \leq \frac{32 \gamma^5 L^4}{\alpha \delta} \sum_{t=0}^{m-1} \mathbb{E} \left\| \bar{\mathbf{v}}_{t}^{s+1} \right\|^2 \left[ \frac{(1+\theta)^{I} - 1}{\theta (\theta - \alpha)} \right] \nonumber \\
        & \ \leq \frac{32 \gamma^5 L^4}{(\delta/2) \delta} \sum_{t=0}^{m-1} \mathbb{E} \left\| \bar{\mathbf{v}}_{t}^{s+1} \right\|^2 \left[ \frac{e-1}{\delta (\delta/2)} \right] \label{eq_phim_V3_3} \\
        & \ = \frac{256 \gamma^5 L^4}{\delta^4} \sum_{t=0}^{m-1} \mathbb{E} \left\| \bar{\mathbf{v}}_{t}^{s+1} \right\|^2. \label{eq_phim_V3}
    \end{align}
    In \eqref{eq_phim_V3_1}, we split the summation over $t$ into blocks of length $I$. For any block $\sum_{t = (c)I}^{(c+1)I-1}$, $\tau(t) = cI$. Therefore, $\ell$ varies from $\ell = [cI+1:t-1]$. Further, over this range of $\ell$, $\tau(\ell) = cI$. Note that in this block, the largest coefficient corresponds to the smallest $j$, i.e., $\mathbb{E} \| \bar{\mathbf{v}}_{cI}^{s+1} \|^2$. We use these upper bounds on coefficients in \eqref{eq_phim_V3_2}. We also use (F3). \eqref{eq_phim_V3_3} follows from (F1), (F2), (F4).   
    \end{proof}
Now, we are ready to prove Lemma \ref{lemma_finite_descent}. In the process, we will utilize the three lemmas derived above.}
\subsubsection{Proof of Lemma \ref{lemma_finite_descent}}    
\begin{proof}{[Lemma \ref{lemma_finite_descent}]}
Substituting the upper bound \eqref{eq_upbd_gradf_avgv_5} in \eqref{eq_upbd_f_avg_Lipschitz} we get
\begin{align}
	& \mathbb{E} f \left( \bar{\mathbf{x}}_{t+1}^{s+1} \right) \nonumber \\
	& \leq \mathbb{E} f \left( \bar{\mathbf{x}}_{t}^{s+1} \right) - \frac{\gamma}{2} \mathbb{E} \big\| \nabla f \left( \bar{\mathbf{x}}_{t}^{s+1} \right) \big\|^2 - \frac{\gamma}{2} (1-L \gamma) \mathbb{E} \big\| \bar{\mathbf{v}}_{t}^{s+1} \big\|^2 \nonumber \\
	& \quad + \frac{\gamma}{2} \frac{4 \gamma^2 L^2}{N^2 B} \sum_{\ell=0}^{t-1} 8 \gamma^2 N L^2 \big( 1 + \mfrac{1}{\delta} \big) \sum_{j=\tau(\ell)}^{\ell-1} (1+\theta)^{\ell-1-j} \mathbb{E} \Big\| \bar{\mathbf{v}}_{j}^{s+1}  \Big\|^2 + \frac{\gamma}{2} \frac{4 \gamma^2 L^2}{N^2 B} N \sum_{\ell = 0}^{t-1} \mathbb{E} \left\| \bar{\mathbf{v}}_{\ell}^{s+1}  \right\|^2 \nonumber \\
	& \quad + \frac{\gamma}{2} 2 E_0^{s+1} + \frac{\gamma}{2} \frac{2 L^2}{N} \left( 1 + \mfrac{1}{\alpha} \right) \gamma^2 \sum_{\ell=\tau(t)+1}^{t-1} (1+\alpha)^{t-1-\ell} 8 \gamma^2 N L^2 \big( 1 + \mfrac{1}{\delta} \big) \sum_{j=\tau(\ell)}^{\ell-1} (1+\theta)^{\ell-1-j} \mathbb{E} \Big\| \bar{\mathbf{v}}_{j}^{s+1} \Big\|^2 \nonumber \\
	&= \mathbb{E} f \left( \bar{\mathbf{x}}_{t}^{s+1} \right) - \frac{\gamma}{2} \mathbb{E} \big\| \nabla f \left( \bar{\mathbf{x}}_{t}^{s+1} \right) \big\|^2 - \frac{\gamma}{2} (1-L \gamma) \mathbb{E} \big\| \bar{\mathbf{v}}_{t}^{s+1} \big\|^2 \nonumber \\
	& \quad + \frac{16 \gamma^5 L^4}{N B} \big( 1 + \mfrac{1}{\delta} \big) \sum_{\ell=0}^{t-1} \sum_{j=\tau(\ell)}^{\ell-1} (1+\theta)^{\ell-1-j} \mathbb{E} \Big\| \bar{\mathbf{v}}_{j}^{s+1}  \Big\|^2 + \frac{2 \gamma^3 L^2}{N B} \sum_{\ell = 0}^{t-1} \mathbb{E} \left\| \bar{\mathbf{v}}_{\ell}^{s+1}  \right\|^2 \nonumber \\
	& \quad + \gamma E_0^{s+1} + 8 \gamma^5 L^4 \left( 1 + \mfrac{1}{\alpha} \right) \big( 1 + \mfrac{1}{\delta} \big) \sum_{\ell=\tau(t)+1}^{t-1} (1+\alpha)^{t-1-\ell} \sum_{j=\tau(\ell)}^{\ell-1} (1+\theta)^{\ell-1-j} \mathbb{E} \Big\| \bar{\mathbf{v}}_{j}^{s+1}  \Big\|^2 \label{eq_upbd_Phi_t}
\end{align}
Note that, as discussed earlier, for finite sum problems, $E_0^{s+1} \triangleq \mathbb{E} \left\| \bar{\mathbf{v}}_{0}^{s+1} - \mfrac{1}{N} \sum\nolimits_{i=1}^N \nabla f_i \left( \mathbf{x}_{i,0}^{s+1} \right) \right\|^2 = 0$. Further, summing \eqref{eq_upbd_Phi_t} over $t=0,\hdots,m-1$, we get
\begin{align}
    \mathbb{E} f \left( \bar{\mathbf{x}}_{m}^{s+1} \right) & \leq \mathbb{E} f \left( \bar{\mathbf{x}}_{0}^{s+1} \right) - \frac{\gamma}{2} \sum_{t=0}^{m-1} \mathbb{E} \big\| \nabla f \left( \bar{\mathbf{x}}_{t}^{s+1} \right) \big\|^2 - \frac{\gamma}{2} (1-L \gamma) \sum_{t=0}^{m-1} \mathbb{E} \big\| \bar{\mathbf{v}}_{t}^{s+1} \big\|^2 \nonumber \\
    & + \frac{16 \gamma^5 L^4}{N B} \big( 1 + \mfrac{1}{\delta} \big) \sum_{t=0}^{m-1} \sum_{\ell=0}^{t-1} \sum_{j=\tau(\ell)}^{\ell-1} (1+\theta)^{\ell-1-j} \mathbb{E} \Big\| \bar{\mathbf{v}}_{j}^{s+1}  \Big\|^2 + \frac{2 \gamma^3 L^2}{N B} \sum_{t=0}^{m-1} \sum_{\ell = 0}^{t-1} \mathbb{E} \left\| \bar{\mathbf{v}}_{\ell}^{s+1}  \right\|^2 \nonumber \\
	& + 8 \gamma^5 L^4 \left( 1 + \mfrac{1}{\alpha} \right) \big( 1 + \mfrac{1}{\delta} \big) \sum_{t=0}^{m-1} \sum_{\ell=\tau(t)+1}^{t-1} \sum_{j=\tau(\ell)}^{\ell-1} (1+\alpha)^{t-1-\ell} (1+\theta)^{\ell-1-j} \mathbb{E} \left\| \bar{\mathbf{v}}_{j}^{s+1} \right\|^2 \label{eq_upbd_Phi_sum_t}
\end{align}
\pk{Substituting the upper bounds derived in Lemmas \ref{lem_finite_descent1}, \ref{lem_finite_descent2} and \ref{lem_finite_descent3}, we get the result of Lemma \ref{lemma_finite_descent}.}
\end{proof}

\subsection{Choice of $\gamma$}
\label{app_gamma_choice}
Suppose $\gamma$ be selected such that
\begin{align}
    L \gamma \leq \frac{1}{8} \ \wedge \ \frac{128 \gamma^{4} L^4 m}{N B \delta^2} \leq \frac{1}{8} \ \wedge \ \frac{4 \gamma^{2} L^2 m}{N B} \leq \frac{1}{8} \ \wedge \ \frac{512 \gamma^{4} L^4}{\delta^4} \leq \frac{1}{8} \label{eq_gamma_choose}
\end{align}
As we shall see in Section \ref{subsec_samp_comp}, the optimal choices $m = I \sqrt{N n}, B = \frac{1}{I} \sqrt{\frac{n}{N}}$. Substituting these values in one of the inequalities above, we get
\begin{align}
    \frac{128 \gamma^{4} L^4 I \sqrt{N n}}{N \frac{1}{I} \sqrt{\frac{n}{N}}} \leq \frac{\delta^2}{8} \quad & \Rightarrow \quad 128 \gamma^{4} L^4 I^2 \leq \frac{1}{8 I^2} \qquad \because \delta < \frac{1}{I} \text{ (F4)} \nonumber \\
    & \Rightarrow \quad \gamma \leq \frac{1}{4 \sqrt{2} L I}.
\end{align}
Repeating a similar reasoning for all the inequalities in \eqref{eq_gamma_choose}, we get
\begin{align}
    \label{eq_upbd_stepsize}
    \gamma \leq \min \left\{ \frac{1}{8L}, \frac{1}{4 \sqrt{2} L I}, \frac{1}{8 L I} \right\} \quad \Rightarrow \quad \gamma \leq \frac{1}{8 L I}.
\end{align}
Therefore, we can choose a constant step size $\gamma$, small enough (and independent of $N, n$), such that \eqref{eq_coeff_avg_v} holds.


\subsection{Proof of Theorem \ref{thm_conv_finite_sum}}
\label{app_thm_conv_finite_sum}
\begin{proof}
We have
\begin{align}
    \min_{s,t} \left[ \mathbb{E} \left\| \nabla f \left( \bar{\mathbf{x}}_{t}^{s+1} \right) \right\|^2 + \frac{1}{N} \sum_{i=1}^N \mathbb{E} \left\| \mathbf{x}_{i,t}^{s+1} - \bar{\mathbf{x}}_{t}^{s+1} \right\|^2 \right]  \leq \frac{1}{T} \sum_{s=0}^{S-1} \sum_{t=0}^{m-1} \left[ \mathbb{E} \left\| \nabla f \left( \bar{\mathbf{x}}_{t}^{s+1} \right) \right\|^2 + \frac{1}{N} \sum_{i=1}^N \mathbb{E} \left\| \mathbf{x}_{i,t}^{s+1} - \bar{\mathbf{x}}_{t}^{s+1} \right\|^2 \right] \label{eq_final_bd_1} 
    \end{align}
    \pk{Upper bounding the second term on the right hand side, we have}
    \begin{align}
    & \frac{1}{T} \sum_{s=0}^{S-1} \sum_{t=0}^{m-1} \frac{1}{N} \sum_{i=1}^N \mathbb{E} \big\| \mathbf{x}_{i,t}^{s+1} - \bar{\mathbf{x}}_{t}^{s+1} \big\|^2 \nonumber \\
    & \quad \leq \frac{1}{T} \sum_{s=0}^{S-1} \sum_{t=0}^{m-1} \frac{1}{N} 8 \gamma^4 N L^2 \big( 1 + \mfrac{1}{\delta} \big) \left( 1 + \mfrac{1}{\alpha} \right) \sum_{\ell=\tau(t)+1}^{t-1} (1+\alpha)^{t-1-\ell} \sum_{j=\tau(\ell)}^{\ell-1} (1+\theta)^{\ell-1-j} \mathbb{E} \left\| \bar{\mathbf{v}}_{j}^{s+1}  \right\|^2 \label{eq_total_avg_x_avgx_1} \\
    & \quad \leq \frac{256 \gamma^4 L^4}{\delta^4} \frac{1}{T} \sum_{s=0}^{S-1} \sum_{t=0}^{m-1} \mathbb{E} \left\| \bar{\mathbf{v}}_{t}^{s+1} \right\|^2 \label{eq_total_avg_x_avgx_2} \\
    & \quad \leq\frac{1}{2} \frac{1}{T} \sum_{s=0}^{S-1}  \sum_{t=0}^{m-1} \mathbb{E} \left\| \bar{\mathbf{v}}_{t}^{s+1} \right\|^2 \label{eq_total_avg_x_avgx_3}
\end{align}
where, \eqref{eq_total_avg_x_avgx_1} follows from Lemma \ref{lemma_finite_nw_error}; \eqref{eq_total_avg_x_avgx_2} follows from the bound in \eqref{eq_phim_V3}; \eqref{eq_total_avg_x_avgx_3} follows from \eqref{eq_coeff_avg_v}. 

\pk{Replacing \eqref{eq_total_avg_x_avgx_3} in \eqref{eq_final_bd_1}, we get}
    \begin{align}
    &  \min_{s,t} \left[ \mathbb{E} \left\| \nabla f \left( \bar{\mathbf{x}}_{t}^{s+1} \right) \right\|^2 + \frac{1}{N} \sum_{i=1}^N \mathbb{E} \left\| \mathbf{x}_{i,t}^{s+1} - \bar{\mathbf{x}}_{t}^{s+1} \right\|^2 \right] \nonumber\\
    & \quad \leq \frac{1}{T} \sum_{s=0}^{S-1} \sum_{t=0}^{m-1} \left[ \mathbb{E} \left\| \nabla f \left( \bar{\mathbf{x}}_{t}^{s+1} \right) \right\|^2 + \frac{1}{2} \mathbb{E} \left\| \bar{\mathbf{v}}_{t}^{s+1} \right\|^2 \right] \nonumber \\
    & \quad  \leq \frac{1}{T} \sum_{s=0}^{S-1} \sum_{t=0}^{m-1} \left[ \mathbb{E} \left\| \nabla f \left( \bar{\mathbf{x}}_{t}^{s+1} \right) \right\|^2 + \left( 1 - L \gamma - \left[ \frac{128 \gamma^{4} L^4 m}{N B \delta^2} + \frac{4 \gamma^{2} L^2 m}{N B} + \frac{512 \gamma^{4} L^4}{\delta^4} \right] \right) \mathbb{E} \left\| \bar{\mathbf{v}}_{t}^{s+1} \right\|^2 \right] \label{eq_final_bd_2} \\
    & \quad   \leq \frac{2 \left( f(\mathbf{x}^0) - f_{\ast} \right)}{T \gamma}. \label{eq_final_bd_3}
\end{align}
Here, \eqref{eq_final_bd_2} follows from \eqref{eq_coeff_avg_v}; \eqref{eq_final_bd_3} follows from \eqref{eq_favg_over_all_epcohs}.
\end{proof}

\section{Online case}
\subsection{Proof of Lemma \ref{lemma_online_bd_var}}
\label{app_lemma_online_bd_var}
\begin{proof} We have from the definition of $E_0^{s + 1}$
\begin{align}
    E_0^{s+1} &= \mathbb{E} \Big\| \bar{\mathbf{v}}_{0}^{s+1} - \mfrac{1}{N} \sum_{i=1}^N \nabla f_i \left( \mathbf{x}_{i,0}^{s+1} \right) \Big\|^2 \nonumber \\
    &= \mathbb{E} \Big\| \mfrac{1}{N} \sum_{i=1}^N \mfrac{1}{n_b} \sum_{\xi_i} \nabla f_i \left( \mathbf{x}_{i,0}^{s+1}; \xi_i \right) - \mfrac{1}{N} \sum_{i=1}^N \nabla f_i \left( \mathbf{x}_{i,0}^{s+1} \right) \Big\|^2 \tag*{(steps \ref{alg2_v_init}, \ref{alg2_v_avg_comp}-\ref{alg2_v_avg_assign} in Algorithm \ref{alg2})} \nonumber \\
    &= \mfrac{1}{N^2} \sum_{i=1}^N \mathbb{E} \Big\| \mfrac{1}{n_b} \sum_{\xi_i} \left\{ \nabla f_i \left( \mathbf{x}_{i,0}^{s+1}; \xi_i \right) - \nabla f_i \left( \mathbf{x}_{i,0}^{s+1} \right) \right\} \Big\|^2 \label{eq_upbd_E0_1} \\
    &= \mfrac{1}{N^2} \sum_{i=1}^N \mfrac{1}{n_b^2} \mathbb{E} \sum_{\xi_i} \Big\| \nabla f_i \left( \mathbf{x}_{i,0}^{s+1}; \xi_i \right) - \nabla f_i \left( \mathbf{x}_{i,0}^{s+1} \right) \Big\|^2 \label{eq_upbd_E0_2} \\
    & \leq \frac{\sigma^2}{N n_b}. \label{eq_upbd_E0_3}
\end{align}
\eqref{eq_upbd_E0_1} follows since for $i \neq j$
\begin{align*}
    & \mathbb{E} \left\langle \sum_{\xi_i} \left\{ \nabla f_i \left( \mathbf{x}_{i,0}^{s+1}; \xi_i \right) - \nabla f_i \left( \mathbf{x}_{i,0}^{s+1} \right) \right\}, \sum_{\xi_j} \left\{ \nabla f_j \left( \mathbf{x}_{j,0}^{s+1}; \xi_j \right) - \nabla f_j \left( \mathbf{x}_{j,0}^{s+1} \right) \right\} \right\rangle \\
    & = \mathbb{E} \left\langle \mathbb{E} \sum_{\xi_i} \left\{ \nabla f_i \left( \mathbf{x}_{i,0}^{s+1}; \xi_i \right) - \nabla f_i \left( \mathbf{x}_{i,0}^{s+1} \right) \right\}, \mathbb{E} \sum_{\xi_j} \left\{ \nabla f_j \left( \mathbf{x}_{j,0}^{s+1}; \xi_j \right) - \nabla f_j \left( \mathbf{x}_{j,0}^{s+1} \right) \right\} \right\rangle = 0.
\end{align*}
This is because, given $\mathbf{x}_{i,0}^{s+1} = \bar{\mathbf{x}}_m^{s}$, the samples at each node are picked uniformly randomly, and independent of other nodes. \eqref{eq_upbd_E0_2} follows since samples at any node are also picked independent of each other. Therefore, for any two distinct samples $\xi_i \neq \zeta_i$,
\begin{align*}
    & \mathbb{E} \left\langle \nabla f_i \left( \mathbf{x}_{i,0}^{s+1}; \xi_i \right) - \nabla f_i \left( \mathbf{x}_{i,0}^{s+1} \right), \nabla f_i \left( \mathbf{x}_{i,0}^{s+1}; \zeta_i \right) - \nabla f_i \left( \mathbf{x}_{i,0}^{s+1} \right) \right\rangle \\
    & = \mathbb{E} \left\langle \mathbb{E}_{\xi_i} \nabla f_i \left( \mathbf{x}_{i,0}^{s+1}; \xi_i \right) - \nabla f_i \left( \mathbf{x}_{i,0}^{s+1} \right), \mathbb{E}_{\zeta_i} \nabla f_i \left( \mathbf{x}_{i,0}^{s+1}; \zeta_i \right) - \nabla f_i \left( \mathbf{x}_{i,0}^{s+1} \right) \right\rangle = 0.
\end{align*}
Finally, \eqref{eq_upbd_E0_3} follows from Assumption (A3).
\end{proof}

\subsection{Proof of Lemma \ref{lemma_online_descent}}
\label{app_lemma_online_descent}
\begin{proof}
Substituting \eqref{eq_upbd_E0_3} in \eqref{eq_upbd_Phi_t} and summing \eqref{eq_upbd_Phi_t} over $t=0,\hdots,m-1$, we get
\begin{align}
    \mathbb{E} f \left( \bar{\mathbf{x}}_{m}^{s+1} \right) & \leq \mathbb{E} f \left( \bar{\mathbf{x}}_{0}^{s+1} \right) - \frac{\gamma}{2} \sum_{t=0}^{m-1} \mathbb{E} \big\| \nabla f \left( \bar{\mathbf{x}}_{t}^{s+1} \right) \big\|^2 - \frac{\gamma}{2} (1-L \gamma) \sum_{t=0}^{m-1} \mathbb{E} \big\| \bar{\mathbf{v}}_{t}^{s+1} \big\|^2 \nonumber \\
    & + \frac{16 \gamma^5 L^4}{N B} \big( 1 + \mfrac{1}{\delta} \big) \sum_{t=0}^{m-1} \sum_{\ell=0}^{t-1} \sum_{j=\tau(\ell)}^{\ell-1} (1+\theta)^{\ell-1-j} \mathbb{E} \Big\| \bar{\mathbf{v}}_{j}^{s+1}  \Big\|^2 + \frac{2 \gamma^3 L^2}{N B} \sum_{t=0}^{m-1} \sum_{\ell = 0}^{t-1} \mathbb{E} \left\| \bar{\mathbf{v}}_{\ell}^{s+1}  \right\|^2 \nonumber \\
	& + 8 \gamma^5 L^4 \left( 1 + \mfrac{1}{\alpha} \right) \big( 1 + \mfrac{1}{\delta} \big) \sum_{t=0}^{m-1} \sum_{\ell=\tau(t)+1}^{t-1} \sum_{j=\tau(\ell)}^{\ell-1} (1+\alpha)^{t-1-\ell} (1+\theta)^{\ell-1-j} \mathbb{E} \left\| \bar{\mathbf{v}}_{j}^{s+1} \right\|^2 \nonumber \\
	& + \frac{\gamma \sigma^2 m}{N n_b} \label{eq_upbd_Phi_online_sum_t}
\end{align}
\eqref{eq_upbd_Phi_online_sum_t} is the same as \eqref{eq_upbd_Phi_sum_t} except the additional last term in \eqref{eq_upbd_Phi_online_sum_t}. \pk{Therefore, again using the upper bounds derived in Lemmas \ref{lem_finite_descent1}, \ref{lem_finite_descent2} and \ref{lem_finite_descent3}} in \eqref{eq_upbd_Phi_online_sum_t}, we get \eqref{eq_lemma_online_descent}.
\end{proof}
\subsection{Proof of Theorem \ref{thm_conv_online_sum}}
\label{app_thm_conv_online_sum}
\begin{proof} We have
\begin{align}
    & \min_{s,t} \left[ \mathbb{E} \left\| \nabla f \left( \bar{\mathbf{x}}_{t}^{s+1} \right) \right\|^2 + \frac{1}{N} \sum_{i=1}^N \mathbb{E} \left\| \mathbf{x}_{i,t}^{s+1} - \bar{\mathbf{x}}_{t}^{s+1} \right\|^2 \right] \nonumber \\
    & \quad \leq \frac{1}{T} \sum_{s=0}^{S-1} \sum_{t=0}^{m-1} \left[ \mathbb{E} \left\| \nabla f \left( \bar{\mathbf{x}}_{t}^{s+1} \right) \right\|^2 + \left( 1 - L \gamma - \left[ \frac{128 \gamma^{4} L^4 m}{N B \delta^2} + \frac{4 \gamma^{2} L^2 m}{N B} + \frac{512 \gamma^{4} L^4}{\delta^4} \right] \right) \mathbb{E} \left\| \bar{\mathbf{v}}_{t}^{s+1} \right\|^2 \right] \label{eq_final_bd_online_2} \\
    & \quad \leq \frac{2 \left( f(\mathbf{x}^0) - f_{\ast} \right)}{T \gamma} + \frac{2 \sigma^2}{N n_b}. \label{eq_final_bd_online_3}
\end{align}
Here, \eqref{eq_final_bd_online_2} follows from \eqref{eq_total_avg_x_avgx_3} and \eqref{eq_coeff_avg_v}; \eqref{eq_final_bd_online_3} follows from \eqref{eq_favg_online_over_all_epcohs}.
\end{proof}
	
\end{document}